\documentclass[12pt]{article}%
\usepackage{amsmath}
\usepackage{amsfonts}
\usepackage{amssymb}
\usepackage{citesort}
\usepackage{graphicx}
\usepackage{subfigure}%

\begin{document}

\newtheorem{theorem}{Theorem}
\newtheorem{acknowledgement}[theorem]{Acknowledgement}
\newtheorem{proposition}[theorem]{Proposition}
\newtheorem{lemma}[theorem]{Lemma}
\newenvironment{proof}[1][Proof]{\noindent\textbf{#1.} }{\ \rule{0.5em}{0.5em}}

\title{A mathematical model and inversion procedure
for Magneto-Acousto-Electric Tomography (MAET)}

\author{Leonid Kunyansky}

\maketitle

\begin{abstract}
Magneto-Acousto-Electric Tomography (MAET), also known as the Lorentz force or
Hall effect tomography, is a novel hybrid modality designed to be a
high-resolution alternative to the unstable Electrical Impedance Tomography.
In the present paper we analyze existing mathematical models of this method,
and propose a general procedure for solving the inverse problem associated
with MAET. It consists in applying to the data one of the algorithms of
Thermo-Acoustic tomography, followed by solving the Neumann problem for the
Laplace equation and the Poisson equation.

For the particular case when the region of interest is a cube, we present an
explicit series solution resulting in a fast reconstruction algorithm. As we
show, both analytically and numerically, MAET is a stable technique yilelding
high-resolution images even in the presence of significant noise in the data.

\end{abstract}

%
%

\section*{Introduction}

Magneto-Acousto-Electric Tomography (MAET) is based on the measurements of the
electrical potential arising when an acoustic wave propagates through
conductive medium placed in a constant magnetic field \cite{Montalibet,Wen}.
The interaction of the mechanical motion of the free charges (ions and/or
electrons) with the magnetic field results in the Lorentz force that pushes
charges of different signs in opposite directions, thus generating Lorentz
currents within the tissue. The goal of this technique coincides with that of
the Electrical Impedance Tomography (EIT): to reconstruct the conductivity of
the tissue from the values of the electric potential measured on the boundary
of the object. EIT is a fast, inexpensive, and harmless modality, which is
potentially very valuable due to the large contrast in the conductivity
between healthy and cancerous tissues~\cite{BB1,Bor02,CIN}. Unfortunately, the
reconstruction problems arising in EIT are known to be exponentially unstable.

MAET is one of the several recently introduced hybrid imaging techniques
designed to stabilize the reconstruction of electrical properties of the
tissues by coupling together ultrasound waves with other physical phenomena.
Perhaps the best known examples of hybrid methods are the Thermo-Acoustic
Tomography (TAT)~\cite{KrugerTAT} and the closely related Photo-Acoustic
modality, PAT~\cite{KrugerPAT,Oraev94}). In the latter methods the amount of
electromagnetic energy absorbed by the medium is reconstructed from the
measurements (on the surface of the object) of acoustic waves caused by the
thermoacoustic expansion (see e.g. \cite{KuKuRev,WangCRC}). Another hybrid
technique, designed to overcome shortcomings of EIT and yield stable
reconstruction of the conductivity is Acousto-Electric Impedance Tomography
(AEIT)~\cite{WangAET}. It couples together acoustic waves and electrical
currents, through the electroacoustic effect (see~\cite{Lavand}). Although
AEIT has been shown, both theoretically and in numerical simulations, to be
stable and capable of yielding high-resolution
images~\cite{AmmariAET,Cap,KuKuAET,KuKuAET1}, the feasibility of practical
reconstructions is still in question due to the extreme weakness of the
acousto-electric effect.

In the present paper we analyze MAET which also aims to reconstruct the
conductivity in a stable fashion. In MAET this goal is achieved by combining
magnetic field, acoustic excitation and electric measurements, coupled through
the Lorentz force. The physical foundations of MAET were established in
\cite{Montalibet} and~\cite{Wen}. In particular, it was shown in
\cite{Montalibet} that if the tissue with conductivity $\sigma(x)$ moves with
velocity $V(x,t)$ within the constant magnetic field $B$, the arising Lorentz
force will generate Lorentz currents $J_{L}(x,t)$ whose intensity and
direction are given (approximately) by the following formula%
\begin{equation}
J_{L}(x,t)=\sigma(x)B\times V(x,t).\label{E:Lorentz}%
\end{equation}

Originally~\cite{Montalibet,Wen} it was proposed to utilize a focused
propagating acoustic pulse to induce electric response from different parts of
the object. In~\cite{HaiderXu} wavepackets of a certain frequency were used in
a physical experiment to reconstruct the current density in a thin slab of a
tissue. Similarly, in~\cite{AmmariMAET} the use of a perfectly focused
acoustic beam was assumed in a theoretical study and in numerical simulations.
However, in the above-quoted works accurate mathematical model(s) of such
beams were not presented. Moreover, the feasibility of focusing a fixed
frequency acoustic beam at an arbitrary point inside the body in a fully 3D
problem is problematic. In a theoretical study~\cite{Roth} the use of plane
waves of varying frequencies was proposed instead of the beams. This is a more
realistic approach; however, the analysis in that work relies on several crude
approximations (the conductivity is assumed to be close to 1, and the electric
field is approximated by the first non-zero term in the multipole expansion).

To summarize, the existing mathematical models of measurements in MAET are of
approximate nature; moreover, some of them contradict to others. For example,
it was found in~\cite{HaiderXu} that if one uses a pair of electrodes to
measure the voltage (difference of the potentials) at two points $a$ and $b$
on the boundary of the body, the result is the integral of the mixed product
of three vectors: velocity $V$, magnetic induction $B$ and the so-called lead
current $J_{ab}$ (the current that would flow in the body if the difference of
potentials were applied at points $a$ and $b$). The approximate model
in~\cite{Roth} implicitly agrees with this conclusion. However,
in~\cite{AmmariMAET} it is assumed that if the pulse is focused at the point
$x$, the measurements will be proportional to the product of the electric
potential $u(x)$ and conductivity $\sigma(x)$ at that point. This assumption
contradicts the previous models; it also seems to be unrealistic since
potential $u(x)$ is only defined up to an arbitrary constant, while the
measurements are completely determined by the physics of the problem.

In the present paper we first derive, starting from equation~(\ref{E:Lorentz}%
), a rigorous and sufficiently general model of the MAET measurements. Next,
we show that if a sufficient amount of data is measured, one can reconstruct,
almost explicitly and in a stable fashion the conductivity of the tissue. For
general domains the reconstruction can be reduced to the solution of the
inverse problem of TAT followed by the solution of the Neumann problem for the
Laplace equation, and a Poisson equation. In the simpler case of a rectangular
domain the reconstruction formulae can be made completely explicit, and the
solution is obtained by summing several Fourier series. In the latter case the
algorithm is fast, i.e. it runs in $O(n^{3}\log n)$ floating point operations
on a $n\times n\times n$ grid. The results of our numerical simulations show
that one can stably recover high resolution images of the conductivity of the
tissues from MAET\ measurements even in the presence of a significant noise in
the data.

\section{Formulation of the problem}

Suppose that the object of interest whose conductivity $\sigma(x)$ we would
like to recover is supported within an open and bounded region $\Omega$ with
the boundary $\partial\Omega$. For simplicity we will assume that $\sigma(x)$
is smooth in $\Omega,$ does not approach 0, and equals 1 in the vicinity of
$\partial\Omega$; the support of $\sigma(x)-1$ lies in some $\Omega_{1}%
\subset\Omega$ and the distance between $\Omega_{1}$ and $\partial\Omega$ is
non-zero. The object is placed in the magnetic field with a constant magnetic
induction~$B$, and an acoustic wave generated by a source lying outside
$\Omega$ propagates through the object with the velocity $V(x,t)$. Then the
Lorentz force will induce Lorentz currents in $\Omega$ given by
equation~(\ref{E:Lorentz}). Throughout the text we assume that the electrical
interactions occur on much faster time scale than the mechanical ones, and so
all currents and electric potentials depend on $t$ only as a parameter. In
addition to Lorentz currents, the arising electrical potential $u(x,t)$ will
generate secondary, Ohmic currents with intensities given by Ohm's law%
\begin{equation}
J_{O}(x,t)=\sigma(x)\nabla u(x,t).\nonumber
\end{equation}
Since there are no sinks or sources of electric charges within the tissues,
the total current $J_{L}(x,t)+J_{O}(x,t)$ is divergence-free%
\begin{equation}
\nabla\cdot(J_{L}+J_{O})=0.\nonumber
\end{equation}
Thus%
\begin{equation}
\nabla\cdot\sigma\nabla u=-\nabla\cdot\left(  \sigma B\times V\right)
.\label{E:inhomo}%
\end{equation}
Since there are no currents through the boundary, the normal component of the
total current $J_{L}(x,t)+J_{O}(x,t)$ vanishes:
\begin{equation}
\frac{\partial}{\partial n}u(z)=-(B\times V(z))\cdot n(z),\qquad z\in
\partial\Omega,\label{E:inhomobc}%
\end{equation}
where $n(z)$ is the exterior normal to $\partial\Omega$ at point $z$.

We will assume that the boundary values of the potential $u(z,t)$ can be
measured at all points $z$ lying on $\partial\Omega$. More precisely, we will
model the measurements by integrating the boundary values with a weight $I(z)$
and thus forming measuring functional $M$ defined by the formula%
\begin{equation}
M(t)=\int\limits_{\partial\Omega}I(z)u(z,t)dA(z), \label{E:funcdef}%
\end{equation}
where $A(z)$ is the standard area element. Weight $I(z)$ can be a function or
a distribution, subject to the restriction
\begin{equation}
\int\limits_{\partial\Omega}I(z)dA(z)=0.\nonumber
\end{equation}
In particular, if one chooses to use $I(z)=\delta(z-a)-\delta(z-b)$, where
$\delta(\cdot)$ is the 2D Dirac delta-function, then $M$ models the two-point
measuring scheme utilized in~\cite{HaiderXu} and~\cite{Roth}.

In order to understand what kind of information is encoded in the values of
$M(t)$ let us consider solution $w_{I}(x)$ of the following divergence
equation
\begin{align}
\nabla\cdot\sigma\nabla w_{I}(x) &  =0,\label{E:homodiv}\\
\frac{\partial}{\partial n}w_{I}(z) &  =I\left(  z\right)  ,\qquad
z\in\partial\Omega.\label{E:homobc}%
\end{align}
(To ensure the uniqueness of the solution of the above boundary value problem
we will require that the integral of $w_{I}(z)$ over $\Omega$ vanishes.) Then
$w_{I}(x)$ equals the electric potential that would be induced in the tissues
by injecting currents $I\left(  z\right)  $ at the boundary $\partial\Omega$.
Let us denote the corresponding currents $\sigma\nabla w_{I}(x)$ by $J_{I}%
(x)$:%
\begin{equation}
J_{I}(x)=\sigma\nabla w_{I}(x).\label{E:current}%
\end{equation}

Let us now apply the second Green's identity to functions $u(x,t)$, $w_{I}%
(x)$, and $\sigma(x)$:%
\begin{equation}
\int\limits_{\Omega}[w_{J}\nabla\cdot(\sigma\nabla u)-u\nabla\cdot
(\sigma\nabla w_{J})]dx=\int\limits_{\partial\Omega}\sigma\left[  w_{J}%
\frac{\partial}{\partial n}u-u\frac{\partial}{\partial n}w_{J}\right]  dA(z).
\label{E:green}%
\end{equation}
By taking into account~(\ref{E:inhomobc}),~(\ref{E:inhomo}),~(\ref{E:homodiv}%
), and~(\ref{E:homobc}), equation~(\ref{E:green}) can be simplified to%
\begin{equation}
-\int\limits_{\Omega}w_{J}\nabla\cdot\left(  \sigma B\times V\right)
dx=\int\limits_{\partial\Omega}\sigma w_{J}\frac{\partial}{\partial
n}udA(z)-M(t).\nonumber
\end{equation}
Further, by integrating the left hand side of the last equation by parts, and
by replacing $\frac{\partial}{\partial n}u$ with expression~(\ref{E:inhomobc})
we obtain%
\begin{equation}
\int\limits_{\Omega}\sigma\nabla w_{J}\cdot\left(  B\times V\right)
dx-\int\limits_{\partial\Omega}w_{J}\sigma(B\times V)\cdot ndA(z)=-\int%
\limits_{\partial\Omega}\sigma w_{J}(B\times V)\cdot ndA(z)+M(t)\nonumber
\end{equation}
or%
\begin{equation}
M(t)=-\int\limits_{\Omega}\sigma\nabla w_{J}\cdot\left(  B\times V\right)
dx=-\int\limits_{\Omega}J_{I}\cdot\left(  B\times V\right)  dx=B\cdot
\int\limits_{\Omega}J_{I}(x)\times V(x,t)dx \label{E:vcrossb}%
\end{equation}
This equation generalizes equation (1) in~\cite{HaiderXu} obtained for the
particular case $I(z)=\delta(z-a)-\delta(z-b)$.

It is clear from equation~(\ref{E:vcrossb}) that MAET measurements recover
some information about currents $J_{I}(x)$. In order to gain further insight
let us assume that the acoustical properties of the medium, such as speed of
sound $c$ and density $\rho,$ are approximately constant within $\Omega$.
(Such approximation usually holds in breast imaging which is one of the most
important potential applications of this and similar modalities). Then the
acoustic pressure $p(x,t)$ within $\Omega$ satisfies the wave equation%
\begin{equation}
\frac{1}{c^{2}}\frac{\partial^{2}}{\partial t^{2}}p(x,t)=\Delta
p(x,t).\nonumber
\end{equation}
Additionally, $p(x,t)$ is the time derivative of the velocity potential
$\varphi(x,t)$ (see, for example~\cite{Colton}), so that
\begin{align}
V(x,t) &  =\frac{1}{\rho}\nabla\varphi(x,t),\label{E:velpoten}\\
p(x,t) &  =\frac{\partial}{\partial t}\varphi(x,t).\nonumber
\end{align}
Velocity potential $\varphi(x,t)$ also satisfies the wave equation%
\begin{equation}
\frac{1}{c^{2}}\frac{\partial^{2}}{\partial t^{2}}\varphi(x,t)=\Delta
\varphi(x,t).\nonumber
\end{equation}
Now, by taking into account~(\ref{E:velpoten}), equation~(\ref{E:vcrossb}) can
be re-written as%
\begin{equation}
M(t)=\frac{1}{\rho}B\cdot\int\limits_{\Omega}J_{I}(x)\times\nabla
\varphi(x,t)dx\nonumber
\end{equation}
Further, by noticing that%
\begin{equation}
\nabla\times\left(  \varphi J_{I}\right)  =-J_{I}\times\nabla\varphi
+\varphi\nabla\times J_{I}\nonumber
\end{equation}
we obtain%
\begin{align}
M(t) &  =\frac{1}{\rho}B\cdot\left[  -\int\limits_{\Omega}\nabla\times\left(
\varphi J_{I}\right)  dx+\int\limits_{\Omega}\varphi(x,t)\nabla\times
J_{I}(x)dx\right]  \nonumber\\
&  =\frac{1}{\rho}B\cdot\left[  \int\limits_{\partial\Omega}\varphi
(z,t)J_{I}(z)\times n(z)dA(z)+\int\limits_{\Omega}\varphi(x,t)\nabla\times
J_{I}(x)dx\right]  .\label{E:surfterm}%
\end{align}

In some situations the above equations can be further simplified. For example,
if at some moment of time $t$ velocity potential $\varphi(x,t)$ vanishes on
the boundary $\partial\Omega$, then the surface integral in~(\ref{E:surfterm})
also vanishes:%
\begin{equation}
M(t)=\frac{1}{\rho}B\cdot\int\limits_{\Omega}\varphi(x,t)\nabla\times
J_{I}(x)dx. \label{E:model}%
\end{equation}
Similarly, if boundary $\partial\Omega$ is located far away from the support
of inhomogeneity of $\sigma(x)$, the surface integral in~(\ref{E:surfterm})
can be neglected, and we again obtain equation~(\ref{E:model}).

Equation~(\ref{E:surfterm}) is our mathematical model of the MAET
measurements. Our goal is to reconstruct from measurements $M(t)$ conductivity
$\sigma(x)$ --- by varying, if necessary, $B$, $\varphi(x,t)$, and $I(z)$.

Our strategy for solving this problem is outlined in the following sections.
However, some conclusions can be reached just by looking at the equation
(\ref{E:surfterm}). For example, one can notice that if three sets of
measurements are conducted with magnetic induction pointing respectively in
the directions of canonical basis vectors $e_{1}$, $e_{2}$, and $e_{3}$, one
can easily reconstruct the sum of integrals in the brackets
in~(\ref{E:surfterm}). Further, if one focuses $\varphi(x,t)$ so that at the
moment $t=0$ it becomes the Dirac $\delta$-function centered at $y$, i.e.
\begin{equation}
\varphi(x,0)=\delta(x-y)\nonumber
\end{equation}
then one immediately obtains the value of $\nabla\times J_{I}$ at the point
$y$ (such a focusing is theoretically possible as explained in the next
section). Thus, by moving the focusing point through the object, one can
reconstruct the curl of $J_{I}$ in all of $\Omega$.

Our model also explains the observation reported in~\cite{HaiderXu} that no
signal is obtained when the acoustic wavepacket is passing through the regions
of the constant $\sigma(x)$. In such regions current $J_{I}$ is a potential
vector field and, therefore, the integral in~(\ref{E:model}) vanishes.

Finally, it becomes clear that an accurate image reconstruction is impossible
if monochromatic acoustic waves of only a single frequency $k$ are used for
scanning, no matter how well they are focused. In this case the spatial
component $\psi$ of $\varphi(x,t)=\psi(x)\exp(ikt)$ is a solution of the
Helmholtz equation%
\begin{equation}
\Delta\psi+k^{2}\psi=0,\nonumber
\end{equation}
and, within $\Omega$ it can be approximated by the plane waves in the form
$\exp(i\lambda\cdot x)$ with $|\lambda|=k$. Let us assume for simplicity that
the electrical boundary is removed to infinity. Then, measuring $M(t)$ given
by equation~(\ref{E:model}) is equivalent to collecting values of the Fourier
transform of $\nabla\times J_{I}(x)$ corresponding to the wave vectors
$\lambda$ lying on the surface of the sphere $|\lambda|=k$ in the Fourier
domain. The spatial frequencies of function $\nabla\times J_{I}(x)$ with wave
vectors that do not lie on this sphere cannot be recovered.

\section{Solving the inverse problem of MAET}

The first step toward the reconstruction of the conductivity is to reconstruct
currents $J_{I}(x)$ corresponding to certain choices of $I(z)$. Let us assume
that all the measurements are repeated three times, with magnetic induction
$B$ pointing respectively in the directions of canonical basis vectors $e_{1}%
$, $e_{2}$, and $e_{3}$. Then, as mentioned above, if $\varphi(x,0)=\delta
(x-y)$, one readily recovers from the measurements $M(0)$ the curl of the
current at $y$, i.e. $\nabla\times J_{I}(y)$. Generating such a velocity
potential is possible at least theoretically. For example, if one
simultaneously propagates plane waves $\varphi_{\lambda}(x,t)=\exp
(i\lambda\cdot x-i|\lambda|t)$ with all possible wave vectors $\lambda$, the
combined velocity potential at the moment $t=0$ will add up to the Dirac
delta-function $\delta(x)$. Such an arrangement is unlikely to be suitable for
a practical implementation: firstly, the sources of sound would have to be
removed far from the object to produce a good approximation to plane waves
within the object. Secondly, the sources would have to completely surround the
object to irradiate it from all possible directions. Finally, all the sources
would have to be synchronized. A variation of this approach is to place small
point-like sources in the vicinity of the object. In this case, instead of
plane waves, spherical monochromatic waves or propagating spherical fronts
would be generated. These types of waves can also be focused into a
delta-function (some discussion of such focusing and a numerical example can
be found in~\cite{KuKuAET1}).

\subsection{Synthetic focusing}

However, a more practical approach is to utilize some realistic measuring
configuration (e.g. one consisting of one or several small sources scanning
the boundary sequentially), and then to synthesize algorithmically from the
realistic data the desired measurements that correspond to the delta-like
velocity potential. Such a \emph{synthetic focusing} was first introduced in
the context of hybrid methods in~\cite{KuKuAET,KuKuAET1,Oberw}. It was shown,
in applications to AET and to the acoustically modulated optical tomography,
that such a synthetic focusing is equivalent to solving the inverse problem of
TAT. The latter problem has been studied extensively, and a wide variety of
methods is known by now (we will refer the reader to
reviews~\cite{KuKuRev,WangCRC} and references therein). The same technique can
be applied to MAET, as explained below.

\subsubsection{Measuring functionals solve the wave equation}

Let us consider a spherical propagating front originated at the point $y$. If
the initial conditions on the pressure $p_{y}(x,t)$ are%
\begin{equation}
\left\{
\begin{array}
[c]{c}%
p_{y}(x,0)=\delta(x-y),\\
\frac{\partial}{\partial t}p_{y}(x,0)=0,
\end{array}
\right. \nonumber
\end{equation}
then $p_{y}(x,t)$ can be represented in the whole of $\mathbb{R}^{3}$ by means
of the Kirchhoff formula~\cite{Vladimirov}
\begin{equation}
p_{y}(x,t)=\frac{\partial}{\partial t}\frac{\delta(|x-y|-ct)}{4\pi
|x-y|}.\nonumber
\end{equation}
Such a front can be generated by a small transducer placed at $y$ and excited
by a delta-like electric pulse; such devices are common in ultrasonic imaging.

Velocity potential $\varphi(x,y,t)$ corresponding to $p_{y}(x,t)$ then equals%
\begin{equation}
\varphi(x,y,t)=\frac{\delta(|x-y|-ct)}{4\pi|x-y|}.\label{E:acfront}%
\end{equation}
The role of variables $x$ and $y$ is clearly interchangeable; $\varphi(x,y,t)$
is the retarded Green's function of the wave equation~\cite{Vladimirov} either
in $x$ and $t$, or in $y$ and $t$. Moreover, consider the following
convolution $H(y,t)$ of a finitely supported smooth function $h(y)$ with
$\varphi$%

\begin{equation}
H(y,t)=\int\limits_{\mathbb{R}^{3}}h(y)\frac{\delta(|x-y|-ct)}{4\pi
|x-y|}dx.\nonumber
\end{equation}
Then $H(y,t)$ is the solution of the following initial value problem (IVP)\ in
$\mathbb{R}^{3}$ \cite{Vladimirov}:
\begin{equation}
\left\{
\begin{array}
[c]{c}%
\frac{1}{c^{2}}\frac{\partial^{2}}{\partial t^{2}}H(y,t)=\Delta_{y}H(y,t)\\
H(y,0)=0,\\
\frac{\partial}{\partial t}H(y,0)=h(y).
\end{array}
\right.  \label{E:wavesyst}%
\end{equation}

Suppose now that a set of MAET measurements is obtained with propagating wave
fronts $\varphi(x,y,t)$ with different centers $y$ (while $I(z)$ and $B$ are
kept fixed). By substituting (\ref{E:acfront}) into (\ref{E:surfterm}) we find
that, for each $y,$ the corresponding measuring functional $M_{I,B}(y,t)$ can
be represented as the sum of two terms:
\begin{equation}
M_{I,B}(y,t)=M_{I,B}^{\mathrm{sing}}(y,t)+M_{I,B}^{\mathrm{reg}}%
(y,t),\nonumber
\end{equation}
where%
\begin{align}
M_{I,B}^{\mathrm{sing}}(y,t) &  =\frac{1}{\rho}\int\limits_{\partial\Omega
}\frac{\delta(|z-y|-ct)}{4\pi|z-y|}B\cdot J_{I}(z)\times
n(z)dA(z),\label{E:realmesing}\\
M_{I,B}^{\mathrm{reg}}(y,t) &  =\frac{1}{\rho}\int\limits_{\Omega}\frac
{\delta(|x-y|-ct)}{4\pi|x-y|}B\cdot\nabla\times J_{I}(x)dx.\label{E:realmereg}%
\end{align}
It is clear from the above discussion that both terms $M_{I,B}^{\mathrm{sing}%
}(y,t)$ and $M_{I,B}^{\mathrm{reg}}(y,t)$ solve the wave equation in
$\mathbb{R}^{3}$, subject to the initial conditions
\begin{align}
\frac{\partial}{\partial t}M_{I,B}^{\mathrm{sing}}(x,0) &  =\frac{1}{\rho
}B\cdot J_{I}(x)\times n(x)\delta_{\partial\Omega}(x),\label{E:initdersing}\\
\frac{\partial}{\partial t}M_{I,B}^{\mathrm{reg}}(x,0) &  =\frac{1}{\rho
}B\cdot\nabla\times J_{I}(x),\label{E:initderreg}\\
M_{I,B}^{\mathrm{sing}}(y,t) &  =M_{I,B}^{\mathrm{reg}}%
(y,t)=0,\label{E:initfun}%
\end{align}
where $\delta_{\partial\Omega}(x)$ is the delta-function supported on
$\partial\Omega$. While singular term $M_{I,B}^{\mathrm{sing}}(y,t)$ solves
the wave equation in the sense of distributions, the regular term
$M_{I,B}^{\mathrm{reg}}(y,t)$ represents a classical solution of the wave equation.

\bigskip

\begin{proposition}
Suppose conductivity $\sigma(x)$ and boundary currents $I(z)$ are $C^{\infty}$
functions of their arguments, and the boundary $\partial\Omega$ is infinitely
smooth. Then the regular part $M_{I,B}^{\mathrm{reg}}(y,t)$ of the measuring
functional is a $C^{\infty}$ solution of the wave equation%
\begin{equation}
\frac{1}{c^{2}}\frac{\partial^{2}}{\partial t^{2}}M_{I,B}^{\mathrm{reg}%
}(y,t)=\Delta_{y}M_{I,B}^{\mathrm{reg}}(y,t),\quad y\in\mathbb{R}^{3},\quad
t\in\lbrack0,\infty),\label{E:waveeqM}%
\end{equation}
satisfying initial conditions (\ref{E:initderreg}) and (\ref{E:initfun}).

\begin{proof}
Under the above conditions, potential $w_{I}(x)$ solving the boundary value
problem (\ref{E:homodiv}), (\ref{E:homobc}) is a $C^{\infty}$ function in
$\Omega$ due to the classical estimates on the smoothness of solutions of
elliptic equations with smooth coefficients \cite{Gilbarg}. Therefore, the
right hand side of (\ref{E:initderreg}) can be extended by zero to a
$C^{\infty}$ function in $\mathbb{R}^{3}.$ Term $M_{I,B}^{\mathrm{reg}}(y,t)$
defined by equation (\ref{E:realmereg}) solves wave equation (\ref{E:waveeqM})
(due to the Kirchhoff formula, see \cite{Vladimirov}) subject to infinitely
smooth initial conditions (\ref{E:initderreg}), (\ref{E:initfun}), and thus it
is a $C^{\infty}$ function for all $y\in\mathbb{R}^{3},\quad t\in
\lbrack0,\infty).$
\end{proof}
\end{proposition}

\subsubsection{Reconstructing the initial conditions}

We would like to reconstruct the right hand side of~(\ref{E:initderreg}) (and,
possibly that of ~(\ref{E:initdersing})) from the measured values of
$M_{I,B}(y,t)$. Since $c$ is assumed constant, in the 3D case  $M_{I,B}(y,t)$
will vanish (due to the Huygens principle) for $t>t_{\max}=D_{\max}/c$ where
$D_{\max}$ is the maximal distance between the points of $\Omega$ and the
acoustic sources. We will assume that $M_{I,B}(y,t)$ is measured for all
$t\in\lbrack0,t_{\max}]$.

The problem of reconstructing $\frac{\partial}{\partial t}M_{I,B}(x,0)$ from
$M_{I,B}(y,t)$ is equivalent to that of reconstructing the initial value
$h(y)$ from the solution of IVP~(\ref{E:wavesyst}). The latter, more general
problem, has been studied extensively in the context of TAT (see, e.g.
\cite{KuKuRev,WangCRC} and references therein). In the case of TAT, points $y$
describe the location of detectors rather than sources, but the rest of the
mathematics remains the same. The most studied situation is when the detectors
are placed on a closed surface $\Sigma$ surrounding the object. If $\Sigma$ is
a sphere, a variety of inversion techniques is known, including (but not
limited to) the explicit inversion formulae~\cite{FPR,Kunyansky,Nguyen},
series techniques \cite{Kun-series,Kun-cyl,Norton2,MXW1}, time reversal by
means of finite differences~\cite{AmbPatch,burg-exac-appro,HKN}, etc. If
$\Sigma$ is a surface of a cube, one can use the inversion
formula~\cite{Kun-cube}, the already mentioned time reversal methods, or the
fast algorithm developed in \cite{Kun-series} for such surfaces.

The choice of the TAT inversion method for application in MAET will depend, in
particular, on the mutual location of the electric boundary $\partial\Omega$
and the acoustic source surface $\Sigma$. One can decide to move the electric
boundary further away by placing the object in a liquid with conductivity
equal to 1 and by submerging the acoustic sources into the liquid, in which
case $\Sigma$ will be inside $\partial\Omega$. Alternatively, one can move the
acoustic surface further away so that it surrounds the electric boundary. And,
finally, one may choose to conduct the electrical measurements on the surface
of the body, and to place the acoustic sources on the same surface, in which
case $\Sigma$ will coincide with $\partial\Omega$.

If $\partial\Omega$ lies inside $\Sigma$, all the above mentioned TAT
inversion techniques (theoretically) will reconstruct both $M_{I,B}%
^{\mathrm{sing}}(y,0)$ and $M_{I,B}^{\mathrm{reg}}(y,0)$. In practice,
accurate numerical reconstruction of the singular term $M_{I,B}^{\mathrm{sing}%
}(y,0)$ supported on $\partial\Omega$ may be difficult to obtain due to the
finite resolution of any realistic measurement system. Luckily, the support of
$M_{I,B}^{\mathrm{reg}}(y,0)$ ($\Omega_{1}$ in our notation) is spatially
separated from $\partial\Omega,$ so that the contributions from the singular
term can be eliminated just by setting the reconstructed image to zero outside
$\Omega_{1}$. As explained further in the paper, $M_{I,B}^{\mathrm{reg}}(y,0)$
contains enough information to reconstruct the conductivity. On the other
hand, $M_{I,B}^{\mathrm{sing}}(y,0)$ does carry some useful information, which
can be recovered by a specialized reconstruction algorithm.

If $\Sigma$ lies inside $\partial\Omega$ (but $\Omega_{1}$ still lies inside
$\Sigma)$, not all TAT reconstruction techniques can be applied. Most
inversion formulae will produce an incorrect result in the presence of the
exterior sources (such as the term $M_{I,B}^{\mathrm{sing}}(y,0)$ supported on
$\partial\Omega$ and exterior with respect to the region enclosed by $\Sigma
)$. On the other hand, formula~\cite{Kun-cube}, series
methods~\cite{Kun-series} and all the time reversal techniques automatically
filter out the exterior sources and thus can be used to reconstruct
$M_{I,B}^{\mathrm{reg}}(y,0).$

The situation is more complicated if surfaces $\Sigma$ and $\partial\Omega$
coincide. However, those methods that are insensitive to sources located
outside $\partial\Omega$ are also insensitive to the sources located on
$\partial\Omega$,  particularly to those  represented by $M_{I,B}%
^{\mathrm{sing}}(y,0).$ These methods can be used to reconstruct
$M_{I,B}^{\mathrm{reg}}(y,0).$ In other words, the following proposition holds:

\begin{proposition}
If values of measuring functional $M_{I,B}(y,t)$ are known for all $y\in
\Sigma$ and $t\in\lbrack0,D_{\max}/c]$, the term $M_{I,B}^{\mathrm{reg}}(y,0)$
can be exactly reconstructed in $\Omega_{1}$ (by using one of the
above-mentioned TAT algorithms). Moreover, if the conditions of Proposition 1
are satisfied, the reconstruction is exact point-wise.
\end{proposition}

\paragraph{Reconstructing the curl}

In order to reconstruct the curl of the current $J_{I}(x),$ we need to repeat
the procedure of finding $M_{I,B}^{\mathrm{reg}}(x,0)$ times, with three
different orientations of $B:$ $B^{(j)}=|B|e_{j}$, $j=1,2,3$. As a result, we
find the projections of $\nabla\times J_{I}(x)$ curl on $e_{1},$ $e_{2},$
$e_{3}$ and thus obtain:%
\begin{equation}
\nabla\times J_{I}(x)=\frac{\rho}{|B|}\sum_{j=1}^{3}e_{j}\frac{\partial
}{\partial t}M_{I,B^{(j)}}^{\mathrm{reg}}(x,0),\quad x\in\Omega_{1}%
.\label{E:curl}%
\end{equation}
(outside $\Omega_{1}$ the curl of $J_{I}(x)$ equals 0 since the conductivity
is constant there).

If, in addition, $\partial\Omega$ lies inside $\Sigma$ and $M_{I,B}%
^{\mathrm{sing}}(y,0)$ has been reconstructed, we obtain the term%
\begin{equation}
J_{I}(x)\times n(x)\delta_{\partial\Omega}(x)=\frac{\rho}{|B|}\sum_{j=1}%
^{3}e_{j}\frac{\partial}{\partial t}M_{I,B^{(j)}}^{\mathrm{sing}%
}(x,0).\label{E:surface}%
\end{equation}

\subsection{Reconstructing the currents}

The considerations of the previous section show how to recover from the values
of the measuring functionals $M_{I,B^{(j)}}(y,t)$ the curl of the current
$J_{I}$ and, in some situations, the surface term (\ref{E:surface}). The next
step is to reconstruct current $J_{I}$ itself.

\subsubsection{General situation}

Let us start with the most general situation and assume that only the curl
$C=\nabla\times J_{I}$ has been reconstructed. Since current $J_{I}$ is a
purely solenoidal field, there exists a vector potential $K(x)$ such that%
\begin{equation}
J_{I}(x)=\nabla\times K(x)+\Psi(x), \label{E:Hdec}%
\end{equation}
where $K(x)$ has the form
\begin{equation}
K(x)=\int\limits_{\Omega}\frac{C(y)}{4\pi(x-y)}dy,\nonumber
\end{equation}
and $\Psi(x)$ is both a solenoidal and potential field. Then there exists
harmonic $\psi(x)$ such that
\begin{equation}
\Psi(x)=\nabla\psi(x). \label{E:Hpot}%
\end{equation}
We know that
\begin{equation}
\left.  J_{I}\cdot n\right\vert _{\partial\Omega}=\left.  \frac{\partial
}{\partial n}w_{I}\right\vert _{\partial\Omega}=I. \label{E:BC}%
\end{equation}
Therefore, by combining equations~(\ref{E:Hdec}) and~(\ref{E:BC}) one obtains
\begin{equation}
\left.  \left(  n\cdot\left(  \nabla\times K\right)  (z)+\frac{\partial
}{\partial n}\psi(z)\right)  \right\vert _{\partial\Omega}=I(z),\nonumber
\end{equation}
and $\psi(z)$ now can be recovered, up to an arbitrary additive constant, by
solving the Neumann problem%
\begin{equation}
\left\{
\begin{array}
[c]{c}%
\Delta\psi(x)=0,\quad x\in\Omega\\
\frac{\partial}{\partial n}\psi(z)=I(z)-n\cdot\left(  \nabla\times
\int\limits_{\Omega}\frac{C(y)}{4\pi(z-y)}dy\right)  ,\quad z\in\partial
\Omega.
\end{array}
\right.  \label{E:Neumann}%
\end{equation}
Now $J_{I}(x)$ is uniquely defined by the formula%
\begin{equation}
J_{I}(x)=\nabla\times\int\limits_{\Omega}\frac{C(y)}{4\pi(x-y)}dy+\nabla
\psi(x),\quad x\in\Omega. \label{E:currentformula}%
\end{equation}

\begin{proposition}
Under smoothness assumption of Proposition~1 current $J_{I}(x)$ is given by
the formula~(\ref{E:currentformula}), where function $\psi(x)$ is the
(classical) solution of the Neumann problem~(\ref{E:Neumann}).
\end{proposition}

\subsubsection{Other possibilities}

If in addition to the curl $\nabla\times J_{I}(x)$, the surface term (equation
(\ref{E:surface})) has been reconstructed, there is no need to solve the
Neumann problem. Instead, function $\Psi(x)$ is given explicitly by the
following formula~\cite{Stewart}:%
\begin{equation}
\Psi(x)=\nabla_{x}\times\int\limits_{\partial\Omega}\frac{J_{I}(y)\times
n(y)}{4\pi(x-y)}dA(y).\nonumber
\end{equation}
The final expression for current $J_{I}$ can now be written as%
\begin{align}
J_{I}(x) &  =\nabla_{x}\times\left[  \int\limits_{\Omega}\frac{C(y)}%
{4\pi(x-y)}dy+\int\limits_{\partial\Omega}\frac{J_{I}(y)\times n(y)}%
{4\pi(x-y)}dA(y)\right]  \nonumber\\
&  =\nabla_{x}\times\int\limits_{\bar{\Omega}}\frac{C(y)+J_{I}(y)\times
n(y)\delta_{\partial\Omega}(y)}{4\pi(x-y)}dy,\quad x\in\Omega,\label{E:fancy}%
\end{align}
where $\bar{\Omega}$ is the closure $\Omega.$ The term with the delta-function
in the numerator of (\ref{E:fancy}) coincides with the surface term given by
equation (\ref{E:surface}). In order to avoid the direct numerical
reconstruction of the singular term, one may want to try to modify the
utilized TAT reconstruction algorithm so as to recover directly the
convolution with $\frac{1}{4\pi|x|}$ contained in equation (\ref{E:fancy}).
The practicality of such an approach requires further investigation.

Finally, in certain simple domains one can find a way to solve the equation%
\begin{equation}
C=\nabla\times J_{I}\nonumber
\end{equation}
for $J_{I}$ in such a way as to explicitly satisfy boundary conditions
(\ref{E:BC}) and thus to avoid the need of solving the Neumann problem
(\ref{E:Neumann}). One such domain is a cube; we present the corresponding
algorithm in Section~\ref{S:cube}.

\subsection{Reconstructing the conductivity}

In order to reconstruct the conductivity we will utilize three currents
$J^{(k)}$, $k=1,2,3$, corresponding to three different boundary conditions
$I_{k}$. As we mentioned before we are using three different magnetic
inductions $B^{(j)}$, $j=1,2,3$. As a result, we obtain the values of the
following measuring functionals $M_{I_{k},B^{(j)}}(y,t)$%
\begin{equation}
M_{I_{k},B^{(j)}}(y,t)=\int\limits_{\partial\Omega}I_{k}(z)u_{(j)}%
(z,t)dA(z),\quad j=1,2,3,\quad k=1,2,3, \label{E:goodfunc}%
\end{equation}
where $u_{(j)}(z,t)$ is the electric potential corresponding to the acoustic
wave with the velocity potential $\varphi(x,z,t)$ propagating through the body
in the presence of constant magnetic field $B^{(j)}$. Notice that the increase
in the number of currents $J^{(k)}$ does not require additional physical
measurements: the same measured boundary values of $u_{(j)}(z,t)$ are used to
compute different measuring functionals by changing the integration weight
$I_{k}(z)$ in equation (\ref{E:funcdef}).

For each of the currents $J^{(k)}$ we apply one of the above-mentioned TAT
reconstruction techniques to compute $\frac{\partial}{\partial t}%
M_{I_{k},B^{(j)}}(y,0)$, The knowledge of the latter functions for $B^{(j)}$,
$j=1,2,3$, allows us to recover the curls $C^{(k)}=\nabla\times J^{(k)}$,
$k=1,2,3$, (equation (\ref{E:curl})) and, possibly, the surface terms
(\ref{E:surface}). Finally, currents $J^{(k)}$ are reconstructed by one of the
methods described in the previous section.

At the first sight, finding $\sigma(x)$ from the knowledge of $J^{(k)}%
=\sigma\nabla w_{I_{k}}$, $k=1,2,3$, is a non-linear problem, since the
unknown electric potentials $w_{I_{k}}$ depend on $\sigma(x)$. However, as
shown below, this problem can be solved explicitly without a linearization or
some other approximation. Indeed, for any $k=1,2,3$, the following formula
holds:%
\begin{equation}
0=\nabla\times\frac{J^{(k)}}{\sigma}=\left(  \nabla\frac{1}{\sigma}\right)
\times J^{(k)}+\frac{1}{\sigma}C^{(k)}=-\frac{1}{\sigma^{2}}\left(
\nabla\sigma\right)  \times J^{(k)}+\frac{1}{\sigma}C^{(k)}\nonumber
\end{equation}
so that%
\begin{equation}
\nabla\ln\sigma(x)\times J^{(k)}(x)=C_{J}^{(k)}(x),\quad x\in\Omega,\quad
k=1,2,3.\nonumber
\end{equation}
Now one can try to find $\nabla\ln\sigma$ at each point in $\Omega$ by solving
the following (in general) over-determined system of linear equations:
\begin{equation}
\left\{
\begin{array}
[c]{c}%
\nabla\ln\sigma\times J^{(1)}=C^{(1)}\\
\nabla\ln\sigma\times J^{(2)}=C^{(2)}\\
\nabla\ln\sigma\times J^{(3)}=C^{(3)}%
\end{array}
\right.  .\label{E:linsys}%
\end{equation}
Let us assume first, that currents $J^{(l)}(x)$, $l=1,2,3$ form a basis in
$\mathbb{R}^{3}$ at each point in~$\Omega$. There are 9 equations in system
(\ref{E:linsys}), whose unknowns are the three components of $\nabla\ln\sigma
$, but the rank of the corresponding matrix does not exceed 6. In order to see
this, let us multiply each equation of (\ref{E:linsys}) by $J^{(l)},l=1,2,3$.
(Since the three currents $J^{(l)}$ form a basis, this is equivalent to a
multiplication by a non-singular matrix). We obtain%
\begin{equation}
\left\{
\begin{array}
[c]{c}%
\nabla\ln\sigma\cdot(J^{(1)}\times J^{(2)})=C^{(1)}\cdot J^{(2)}\\
\nabla\ln\sigma\cdot(J^{(1)}\times J^{(2)})=-C^{(2)}\cdot J^{(1)}\\
\nabla\ln\sigma\cdot(J^{(1)}\times J^{(3)})=C^{(1)}\cdot J^{(3)}\\
\nabla\ln\sigma\cdot(J^{(1)}\times J^{(3)})=-C^{(3)}\cdot J^{(1)}\\
\nabla\ln\sigma\cdot(J^{(2)}\times J^{(3)})=C^{(2)}\cdot J^{(3)}\\
\nabla\ln\sigma\cdot(J^{(2)}\times J^{(3)})=-C^{(3)}\cdot J^{(2)}%
\end{array}
\right.  .\nonumber
\end{equation}
In the case of perfect measurements the right hand sides of equations number
2, 4, and 6 in the above system would coincide with those of equations 1, 3,
and 5, respectively, and therefore the even-numbered equations could just be
dropped from the system. However, in the presence of noise it is better to
take the average of the equations with identical left sides, which is
equivalent to finding the least squares solution of this system. We thus
obtain:%
\begin{equation}
\left\{
\begin{array}
[c]{c}%
\nabla\ln\sigma\cdot(J^{(1)}\times J^{(2)})=\frac{1}{2}(C^{(1)}\cdot
J^{(2)}-C^{(2)}\cdot J^{(1)})\\
\nabla\ln\sigma\cdot(J^{(1)}\times J^{(3)})=\frac{1}{2}(C^{(1)}\cdot
J^{(3)}-C^{(3)}\cdot J^{(1)})\\
\nabla\ln\sigma\cdot(J^{(2)}\times J^{(3)})=\frac{1}{2}(C^{(2)}\cdot
J^{(3)}-C^{(3)}\cdot J^{(2)})
\end{array}
\right.  .\label{E:leastsq}%
\end{equation}
After some simple linear algebra transformations (see Appendix) the solution
of (\ref{E:leastsq}) can be written explicitly as follows:%
\begin{equation}
\nabla\ln\sigma=\frac{1}{2J^{(1)}\cdot(J^{(2)}\times J^{(3)})}M\left(
\begin{array}
[c]{c}%
C^{(2)}\cdot J^{(3)}-C^{(3)}\cdot J^{(2)}\\
-C^{(1)}\cdot J^{(3)}+C^{(3)}\cdot J^{(1)}\\
C^{(1)}\cdot J^{(2)}-C^{(2)}\cdot J^{(1)}%
\end{array}
\right)  ,\label{E:gradsol}%
\end{equation}
where $M$ is $(3\times3)$ matrix whose columns are the Cartesian coordinates
of the currents $J^{(l)},l=1,2,3$:%
\begin{equation}
M=\left(  J^{(1)}\rule[-10pt]{1pt}{26pt}J^{(2)}\rule[-10pt]{1pt}{26pt}%
J^{(3)}\right)  .\label{E:matrix}%
\end{equation}
Since, by assumption, currents $J^{(l)}$ form a basis at each point
of~$\Omega$, the denominator in (\ref{E:gradsol}) never vanishes and, thus,
equation (\ref{E:gradsol}) can be used to reconstruct $\nabla\ln\sigma$ in all
of~$\Omega$. Finally, we compute the divergence of both sides in
(\ref{E:gradsol}):%
\begin{equation}
\Delta\ln\sigma=\frac{1}{2}\nabla\cdot\left[  \frac{1}{J^{(1)}\cdot
(J^{(2)}\times J^{(3)})}M\left(
\begin{array}
[c]{c}%
C^{(2)}\cdot J^{(3)}-C^{(3)}\cdot J^{(2)}\\
-C^{(1)}\cdot J^{(3)}+C^{(3)}\cdot J^{(1)}\\
C^{(1)}\cdot J^{(2)}-C^{(2)}\cdot J^{(1)}%
\end{array}
\right)  \right]  ,\label{E:finalsys}%
\end{equation}
and solve the above Poisson equation for $\ln\sigma$ in $\Omega$ subject to
the Dirichlet boundary conditions
\begin{equation}
\left.  \ln\sigma\right\vert _{\partial\Omega}=0.\label{E:finalbc}%
\end{equation}

The above reconstruction procedure works if currents $J^{(1)}$, $J^{(2)}$, and
$J^{(3)}$ are linearly independent at each point in~$\Omega$. For an arbitrary
conductivity $\sigma$ this cannot be guaranteed. There exists a
counterexample~\cite{Cap} describing such a conductivity for which a boundary
condition can be found such that the corresponding current vanishes at a
certain point within the domain. Clearly, while such a situation can occur, it
is unlikely to occur for an arbitrary conductivity $\sigma$, and our method
should still be useful in practice.

Moreover, the condition of the three currents forming a basis at each point in
space can be relaxed. Below we show that if only one of the currents $J^{(1)}%
$, $J^{(2)}$, and $J^{(3)}$ vanishes (say, $J^{(3)}=0$) at some point and the
two other currents are not parallel, the following truncated system is still
uniquely solvable:%
\begin{equation}
\left\{
\begin{array}
[c]{c}%
\nabla\ln\sigma\times J^{(1)}=C^{(1)}\\
\nabla\ln\sigma\times J^{(2)}=C^{(2)}%
\end{array}
\right.  .\label{E:smallsys}%
\end{equation}
Indeed, let us multiply via dot product the above equations by $J^{(2)}$ and
$J^{(1)}$ respectively, and subtract them. We obtain%
\begin{equation}
\nabla\ln\sigma\cdot(J^{(1)}\times J^{(2)})=\frac{1}{2}\left(  C^{(1)}\cdot
J^{(2)}-C^{(2)}\cdot J^{(1)}\right)  .\label{E:new1}%
\end{equation}
Now, multiply the first equation in (\ref{E:smallsys}) by $J^{(1)}\times
J^{(2)}$. The left hand side will take the form%
\begin{align*}
\left(  \nabla\ln\sigma\times J^{(1)}\right)  \cdot(J^{(1)}\times J^{(2)}) &
=\nabla\ln\sigma\cdot\left[  J^{(1)}\times(J^{(1)}\times J^{(2)})\right]  \\
&  =\nabla\ln\sigma\cdot\left[  \left(  J^{(1)}\cdot J^{(2)}\right)
J^{(1)}-\left(  J^{(1)}\cdot J^{(1)}\right)  J^{(2)}\right]  ,
\end{align*}
which leads to the equation%
\begin{equation}
\nabla\ln\sigma\cdot\left[  \left(  J^{(1)}\cdot J^{(2)}\right)
J^{(1)}-\left(  J^{(1)}\cdot J^{(1)}\right)  J^{(2)}\right]  =C^{(1)}%
\cdot(J^{(1)}\times J^{(2)}).\label{E:new2}%
\end{equation}
Similarly, by multiplying the second equation in (\ref{E:smallsys}) by
$(J^{(2)}\times J^{(1)})$ we obtain%
\begin{equation}
\nabla\ln\sigma\cdot\left[  \left(  J^{(1)}\cdot J^{(2)}\right)
J^{(2)}-\left(  J^{(2)}\cdot J^{(2)}\right)  J^{(1)}\right]  =C^{(2)}%
\cdot(J^{(2)}\times J^{(1)}).\label{E:new3}%
\end{equation}
Equations (\ref{E:new1}), (\ref{E:new2}) and (\ref{E:new3}) form a linear
system with three equations and three unknowns. In other to show that the
matrix of this system is non-singular, it is enough to show that the vectors
given by the bracketed expressions in (\ref{E:new2}) and (\ref{E:new3}) are
not parallel. The cross-product of these terms yields%
\begin{equation}
\lbrack...]\times\lbrack...]=(J^{(1)}\times J^{(2)})\left[  \left(
J^{(1)}\cdot J^{(2)}\right)  ^{2}-\left(  J^{(1)}\cdot J^{(1)}\right)  \left(
J^{(2)}\cdot J^{(2)}\right)  \right]  .\nonumber
\end{equation}
The above expression is clearly non-zero if $J^{(1)}$ and $J^{(2)}$ are not
parallel, and therefore the system of the three equations (\ref{E:new1}),
(\ref{E:new2}), (\ref{E:new3}) is uniquely solvable in this case.

\begin{theorem}
Suppose that the conditions of Proposition 1 are satisfied, and that the
conductivity $\sigma(x)$ and boundary currents $I_{k},$ $k=1,2,3,$ are such
that at each point $x\in\Omega$ two of the three correspondent currents
$J^{(k)}$ are non-parallel. Then the logarithm of the conductivity $\ln\sigma$
is uniquely determined by the values of the measuring functionals
$M_{I_{k},B^{(j)}}(y,t),$ $k=1,2,3,$ $j=1,2,3,$ $y\in\Sigma$ and $t\in
\lbrack0,D_{\max}/c].$
\end{theorem}

\begin{proof}
By Propositions 1, 2, 3, and 4, from the values of $M_{I_{k},B^{(j)}}(y,t)$
one can reconstruct currents $J^{(k)}(x),$ $k=$ $1,2,3,$ at each point $x$ in
$\Omega.$ Since at each point at least two of the three currents (lets call
them $J^{(1)},$ and $J^{(2)}$) are not parallel, system of the three equations
(\ref{E:new1}), (\ref{E:new2}), (\ref{E:new3}) is uniquely solvable, and
$\nabla\ln\sigma$ can be found at each point in $\Omega.$ Since it was assumed
that $\sigma$ is bounded away from zero, $\nabla\ln\sigma$ is a $C^{\infty}$
function in $\Omega.$ By computing the divergence of $\nabla\ln\sigma$ the
problem of finding $\ln\sigma$ reduces to solving the Poisson problem with
zero Dirichlet boundary conditions in a smooth domain $\Omega.$
\end{proof}

The condition that out of the three currents two are non-zero and not
parallel, is less restrictive than the requirement that the three currents are
linearly independent. To the best of our knowledge, there exists no
counterexample showing that the former condition can be violated , i.e. that
one of the three currents generated by linearly independent boundary
conditions vanishes and the other two are parallel at some point in space. On
the other hand, we know of no proof that this cannot happen.

\section{The case of a rectangular domain}

\label{S:cube}

In the previous section we presented a theoretical scheme for finding the
currents $J^{(k)}$, $k=1,2,3$, and conductivity $\sigma(x)$ from the MAET
measurements in a rather general setting, where the electrical domain $\Omega$
and surface $\Sigma$ supporting the acoustic sources are quite arbitrary. This
scheme consists of several steps including the solution of the inverse problem
of TAT in the domain surrounded by $\Sigma$, solution of the Neumann problem
for the Laplace equation in $\Omega$, and solution of the Poisson equation in
$\Omega$ for $\ln\sigma$. All these problems are well-studied and various
algorithms for their numerical solution are well known. However, in the
simplest case when $\Omega$ is a cube and $\Sigma$ coincides with
$\partial\Omega$, the reconstruction can be obtained by means of an explicit
series solution as described below.

\subsection{Fast reconstruction algorithm}

Let us assume that the domain $\Omega$ is a cube $[0,1]\times\lbrack
0,1]\times\lbrack0,1]$, and the sound sources are located on $\partial\Omega$.
(In practice such a measuring configuration will occur if the object is placed
in a cubic tank filled with conductive liquid, and the sound sources and
electrical connections are placed on the tank walls). We will use three
boundary conditions $I_{k}$ defined by the formulae:%
\begin{equation}
I_{k}(x)=\left\{
\begin{array}
[c]{cc}%
\frac{1}{2}, & x\in\partial\Omega,\quad x_{k}=1\\
-\frac{1}{2}, & x\in\partial\Omega,\quad x_{k}=0\\
0, & x\in\partial\Omega,\quad0<x_{k}<1
\end{array}
\right.  ,\quad k=1,2,3,\quad x=(x_{1},x_{2},x_{3}). \label{E:currents}%
\end{equation}
As before, all the measurements are repeated with three different direction of
the magnetic field $B^{(j)}=|B|e_{j}$, $j=1,2,3$, and the values of
functionals $M_{I_{k},B^{(j)}}(y,t)$ (see equation (\ref{E:goodfunc})) are
computed from the measurements of the electrical potentials $u^{(j)}(x,t)$ on
$\partial\Omega$, for $t\in\left[  0,\frac{\sqrt{3}}{c}\right]  $.

We start the reconstruction by applying to $M_{I_{k},B^{(j)}}(y,t)$ the fast
cubic-domain TAT algorithm~\cite{Kun-series} to recover the regular terms
$M_{I^{(k)},B^{(j)}}^{\mathrm{reg}}(x,0),$ $j=1,2,3,$ $k=1,2,3.$ (The
algorithm we chose automatically sets to zero the sources corresponding to the
surface terms $\frac{\partial}{\partial t}M_{I^{(k)},B^{(j)}}^{\mathrm{sing}%
}(x,0)$ supported on $\Sigma)$. This is done for all three directions of
$B^{(j)}$, so that we immediately obtain the curls $C^{(k)}(x)=\nabla\times
J^{(k)}(x)$, $k=1,2,3$, $x\in\partial\Omega$.

Now, since the currents are divergence-free,%
\begin{equation}
\nabla\times C^{(k)}(x)=\nabla\times\nabla\times J^{(k)}(x)=-\Delta
J^{(k)}(x),\nonumber
\end{equation}
and we can try to solve the above equation as a set of Poisson problems for
the components $J^{(k)}(x)$ in $\Omega$. Below we discuss how to enforce the
correct boundary conditions for these Poisson problems.

Let us recall that%
\begin{equation}
J^{(k)}(x)=\sigma\nabla w^{(k)}(x),\nonumber
\end{equation}
where $w^{(k)}(x)$ is the corresponding electric potential. It is convenient
to subtract the linear component of potentials, i.e. to introduce potentials
$w^{(k),0}(x)$ and $J^{(k),0}(x)$ defined as follows:%
\begin{align*}
w^{(k)}(x) &  =w^{(k),0}(x)+x_{k},\\
J^{(k)}(x) &  =J^{(k),0}(x)+e_{k},\quad k=1,2,3.
\end{align*}
Now each $w^{(k),0}(x)$ satisfies zero Neumann conditions on $\partial\Omega$,
and thus can be extended by even reflections to a periodic function in
$\mathbb{R}^{3}$. Since $w^{(k),0}(x)$ is a harmonic function near
$\partial\Omega$, such an extension would be a harmonic function in the
neighborhood of $\partial\Omega$ and its reflections. Therefore, each
$w^{(k),0}(x)$ can be expanded in the Fourier cosine series in~$\Omega$, and
the derivatives of this series yield correct behavior of the so-computed
$\nabla w^{(k),0}(x)$ on $\partial\Omega$, and also of currents $J^{(k),0}(x)$
(the latter currents coincide with $\nabla w^{(k),0}(x)$ in the vicinity of
$\partial\Omega$). Therefore, the components of currents $J^{(k),0}(x)$ should
be expanded in the Fourier series whose basis functions are the corresponding
derivatives of the cosine series. In other words, the following series yield
correct boundary conditions when used as a basis for expanding currents
$J^{(k),0}(x)$, $k=1,2,3$:%

\begin{equation}
\left\{
\begin{array}
[c]{c}%
J_{1}^{(k),0}(x)=\sum\limits_{l=1}^{\infty}\sum\limits_{m=0}^{\infty}%
\sum\limits_{n=0}^{\infty}A_{l,m,n}^{(k),1}\sin\pi lx_{1}\cos\pi mx_{2}\cos\pi
nx_{3}\\
J_{2}^{(k),0}(x)=\sum\limits_{l=0}^{\infty}\sum\limits_{m=1}^{\infty}%
\sum\limits_{n=0}^{\infty}A_{l,m,n}^{(k),2}\cos\pi lx_{1}\sin\pi mx_{2}\cos\pi
nx_{3}\\
J_{3}^{(k),0}(x)=\sum\limits_{l=0}^{\infty}\sum\limits_{m=0}^{\infty}%
\sum\limits_{n=1}^{\infty}A_{l,m,n}^{(k),3}\sin\pi lx_{1}\cos\pi mx_{2}\cos\pi
nx_{3}%
\end{array}
\right.  ,\label{E:series}%
\end{equation}
\bigskip where $J^{(k),0}(x)=\left(  J_{1}^{(k),0}(x),J_{2}^{(k),0}%
(x),J_{3}^{(k),0}(x)\right)  $. Now, since $\nabla\times J^{(k),0}%
(x)=\nabla\times J^{(k)}(x)$:%
\begin{equation}
\Delta J^{(k),0}(x)=-\nabla\times C^{(k)}(x),\quad k=1,2,3\label{E:lapgrad}%
\end{equation}
and, if the Poisson problems~(\ref{E:lapgrad}) are solved for each component
of $J^{(k),0}$ using sine/cosine Fourier series (\ref{E:series}), the correct
boundary conditions will be attained. The computation can be performed
efficiently using the Fast Fourier sine and cosine transforms (FFST and FFCT).

One can notice that before Poisson problems (\ref{E:lapgrad}) can be solved,
the curl of the $C^{(k)}(x)$ needs to be computed. We compute it by expanding
$C^{(k)}(x)$ in the Fourier sine series and by differentiating the series,
again using FFST and FFCT. This computation amounts to numerical
differentiation of the data that may contain a significant amount of noise.
However, this does not give rise to instabilities, since this differentiation
is immediately followed by an inverse Laplacian (describing solution of the
Poisson problem), so that the combined operator is actually smoothing.

Finally, once the currents are reconstructed, we form the right hand side of
the system (\ref{E:finalsys}) and solve this Poisson problem for $\ln
\sigma(x)$ by using the Fourier sine series which yields the desired boundary
conditions (\ref{E:finalbc}). Again, FFST and FFCT are utilized here and in
the computations of the divergence needed to form equation (\ref{E:finalsys}).

\begin{figure}[t]
\begin{center}%
\begin{tabular}
[c]{cc}%
\includegraphics[width=2.3in,height=2.3in]{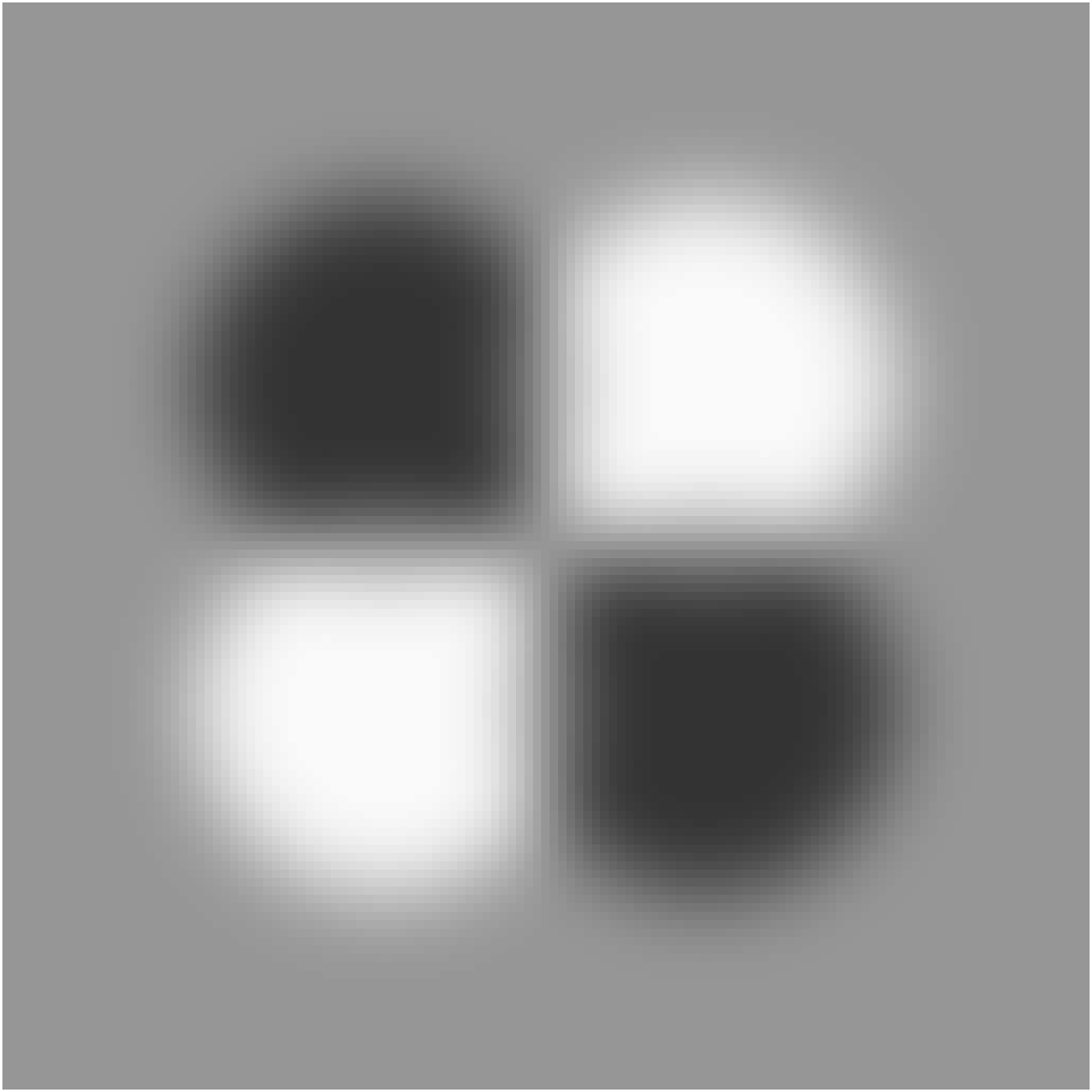} &
\includegraphics[width=2.3in,height=2.3in]{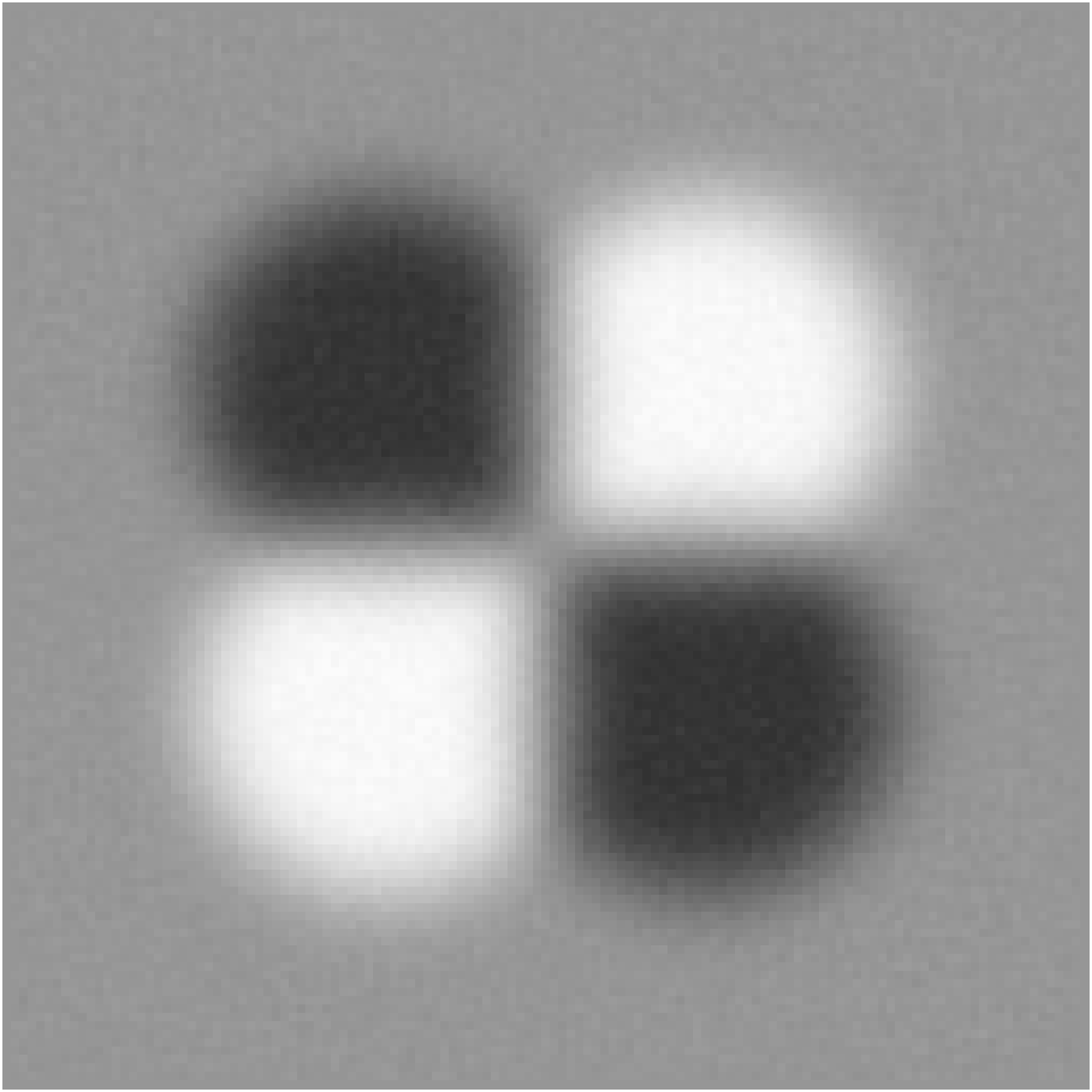}\\
(a) & (b)
\end{tabular}
\end{center}
\caption{A $3D$ simulation with a smooth phantom representing $\ln\sigma(x)$
\newline(a) the cross section of the phantom by the plane $x_{3}=0.5$
\newline(b) reconstruction from the data with added $50\%$ (in $L_{2}$ sense)
noise}%
\label{F:recsmoo}%
\end{figure}

\begin{figure}[t]
\begin{center}
\includegraphics[width=3.0in,height=1.4in]{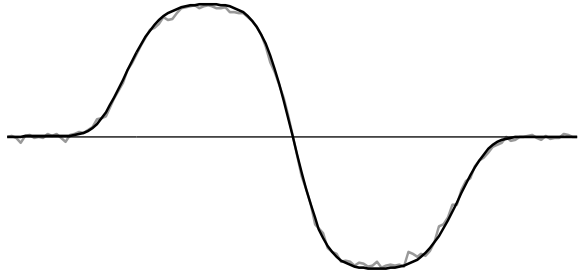}
\end{center}
\caption{The cross section of the reconstructed image (shown in
Figure~\ref{F:recsmoo}) by the line $x_{1}=0.25$, $x_{3}=0.5$. The thick black
line represents the phantom, the gray line corresponds to the reconstructed
image}%
\label{F:smoprof}%
\end{figure}

One could notice that all the steps of the present algorithm are explicit, and
can be performed using FFST and FFCT, so that the number of floating point
operations (flops)\ required to complete the computations is $O(n^{3}\ln n)$
for a $(n\times n\times n)$ computational grid. The same is true for the TAT
reconstruction technique~\cite{Kun-series} used on the first step of
computations, so that the whole method is fast; it has complexity of
$O(n^{3}\ln n)$ flops.

\subsection{Numerical results}

Since the author does not have at his disposal any real MAET\ measurements,
the work of the present reconstruction algorithm (for a cubic domain) will be
demonstrated on simulated data. The most thorough simulation would entail
solving equation (\ref{E:inhomo}) with the boundary conditions
(\ref{E:inhomobc}) for various velocity fields $V(x,t)$ corresponding to the
propagating spherical fronts originating at different locations on
$\partial\Omega$. For a good reconstruction, the number of the data points
should be comparable with the number of unknowns. We would like to reconstruct
the conductivity on a fine 3D grid (say, of the size $257\times257\times257$),
which implies having several millions of unknowns. Thus, we would need to
solve equation (\ref{E:inhomo}) several million times. This task is too
challenging computationally.

Instead, for a given phantom and for the set of functions $I^{(k)}$ given by
equations (\ref{E:currents}) we computed currents $J^{(k)}$ by solving
equation (\ref{E:homodiv}) with boundary conditions (\ref{E:homobc}). Next,
the wave equation (\ref{E:wavesyst}) with the initial condition $h(y)=B^{(j)}%
\cdot\nabla\times J^{(k)}(x)$ was solved for $j=1,2,3$, $k=1,2,3$; the values
of the solution of this equation at points in $\partial\Omega$ simulated the
regular part $\rho M_{I,B}^{\mathrm{reg}}(x,t)$ of the measuring functionals
$\rho M_{I,B}(y,t)$. Note that for simplicity we did not model the singular
term $M_{I,B}^{\mathrm{sing}}(x,t)$ given by (\ref{E:realmesing}).
Theoretically, when the TAT\ reconstruction algorithm~\cite{Kun-series} is
applied to such data, since this term is not supported in $\Omega$ it will not
contribute to the reconstruction, and so the additional effort to simulate it
would not be rewarded with additional insight. \begin{figure}[t]
\begin{center}
\includegraphics[width=3.0in,height=1.7in]{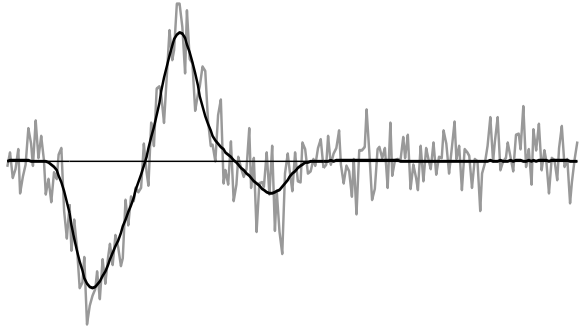}
\end{center}
\caption{Plot of one of the simulated measurement functionals for one of the
detector positions (see text for details): thick black line represenents
accurate values, gray line shows the data with added $50\%$ noise}%
\label{F:sigprof}%
\end{figure}

\begin{figure}[h]
\begin{center}
\includegraphics[width=3.0in,height=1.7in]{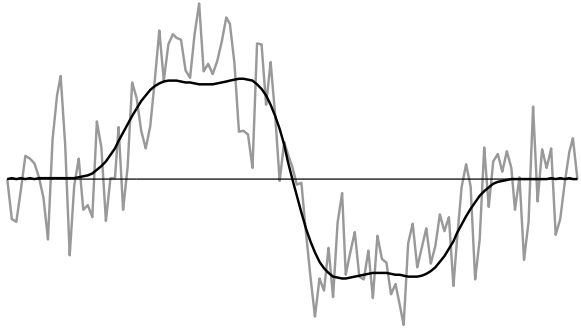}
\end{center}
\caption{ Profile of the reconstructed curl (gray line) compared to the
accurate values (thick black line). Shown are the values of the third
component of $C^{(1)}$ along the line $x_{2}=0.5$, $x_{3}=0.5$ }%
\label{F:crlprof}%
\end{figure}

\begin{figure}[t]
\begin{center}
\includegraphics[width=3.0in,height=2.6in]{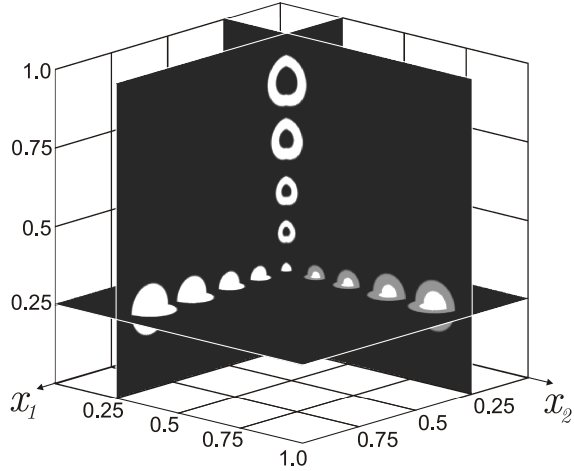}
\end{center}
\caption{The second, (almost) piece-wise constant $3D$ phantom. Shown are
cross sections by the planes $x_{j}=0.25$, $j=1,2,3$}%
\label{F:3Dphan}%
\end{figure}
As the first phantom simulating $\ln\sigma(x)$ we used the following linear
combinations of four $C^{8}$ radially symmetric functions $\varphi(x-x^{(i)})$
with centers $x^{(i)}$ lying in the plane$~x_{3}=0:$%
\begin{align*}
f(x)  &  =\sum_{i=1}^{4}a_{i}\varphi(|x-x^{(i)}|),\\
x^{(1)}  &  =(0.25,0.25,0),\quad a_{1}=0.5,\\
x^{(2)}  &  =(0.25,0.75,0),\quad a_{2}=-0.5,\\
x^{(3)}  &  =(0.75,0.25,0),\quad a_{3}=-0.5,\\
x^{(4)}  &  =(0.75,0.75,0),\quad a_{4}=0.5,
\end{align*}
\begin{figure}[h!]
\begin{center}
\includegraphics[width=4.0in,height=1.7in]{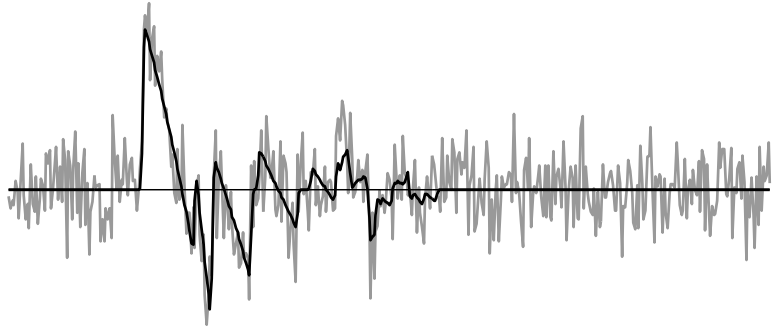}
\end{center}
\caption{Plot of the values of one of the simulated measurement functionals
for one of the detector positions (see text for details): thick black line
represenents accurate values, gray line shows the data with added $100\%$
noise}%
\label{F:noiseprof}%
\end{figure}
where $\varphi(t)$ is a decreasing non-negative trigonometric polynomial on
$[0,r_{0}]$, such that $\varphi(0)=1$, $\varphi(t)=0$ for $t\geq r_{0}$ and
the first eight derivatives of $\varphi$ vanish at $0$ and at $r_{0}$. Radius
$r_{0}$ was equal to 0.34 in this simulation. A gray scale picture of this
phantom is shown in Figure~\ref{F:recsmoo}(a). Figure~\ref{F:recsmoo}(b)
demonstrates the cross-section by the plane $x_{3}=0.5$ of the image
reconstructed on a $129\times129\times129$ computational grid from simulated
MAET data with added simulated noise. The acoustic sources were located at the
nodes of $129\times129$ Cartesian grids on each of the six faces of cubic
domain $\Omega$. For each source 223 values of each of the measuring
functionals were computed, representing 223 different time samples or,
equivalently, 223 different radii of the propagating acoustic front.

The measurement noise was simulated by adding values of uniformly distributed
random variable to the data. The so-simulated noise was scaled in such a way
that for each time series (one source position) the noise intensity in $L^{2}$
norm was 50\% of the intensity of the signal (i.e. of the $L^{2}$ norm of the
data sequence representing the measuring functional). In spite of such high
level of noise in the data, the reconstructed image shown
in~Figure~\ref{F:recsmoo}(b) contains very little noise. This can also be
verified by looking at the plot of the cross section of the latter image along
the line $x_{2}=0.25$, presented in Figure~\ref{F:smoprof}.
\begin{figure}[h]        
\begin{center}%
\begin{tabular}
[c]{cc}%
\includegraphics[width=2.3in,height=2.3in]{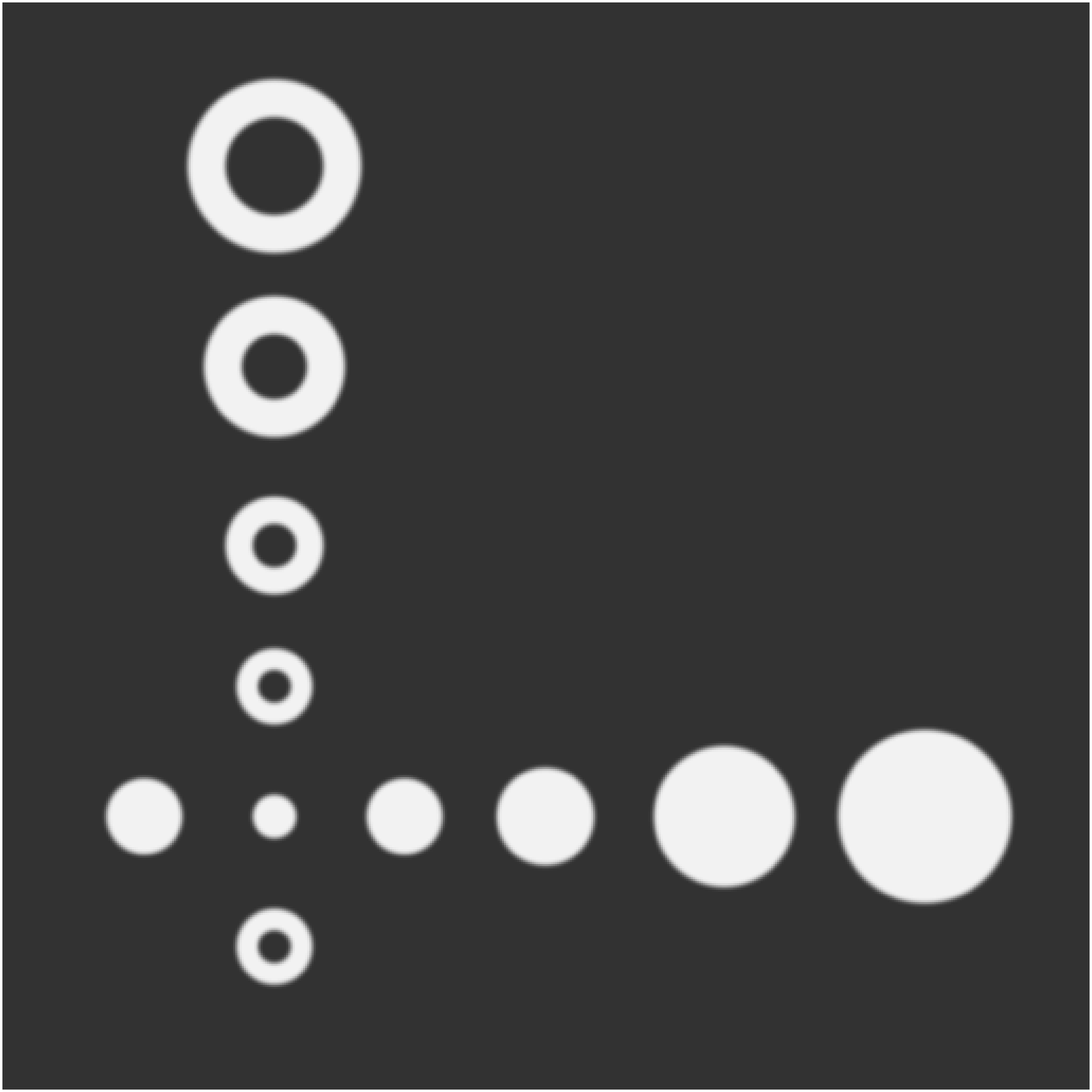} &
\includegraphics[width=2.3in,height=2.3in]{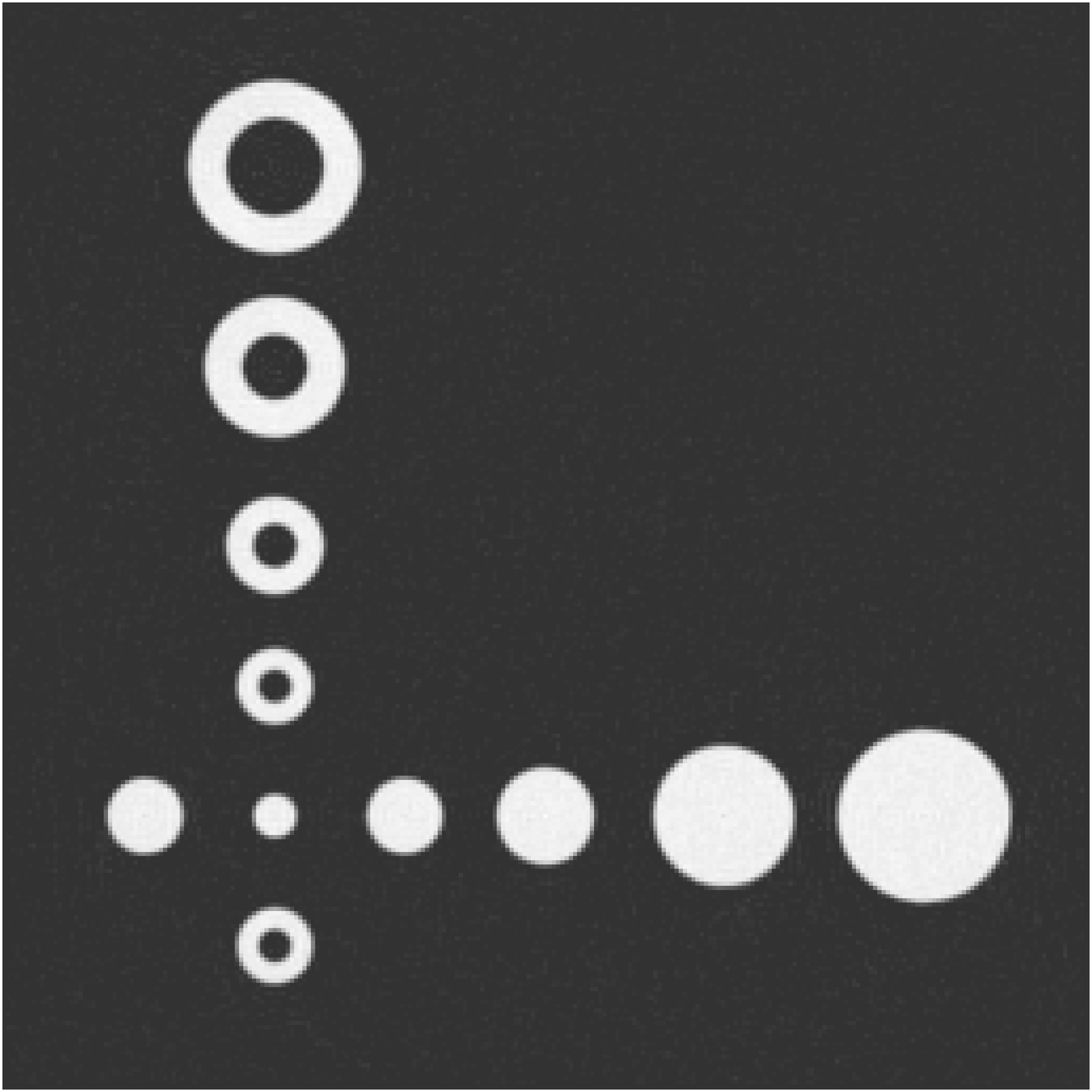}\\
(a) & (b)\\
& \\
\includegraphics[width=2.3in,height=2.3in]{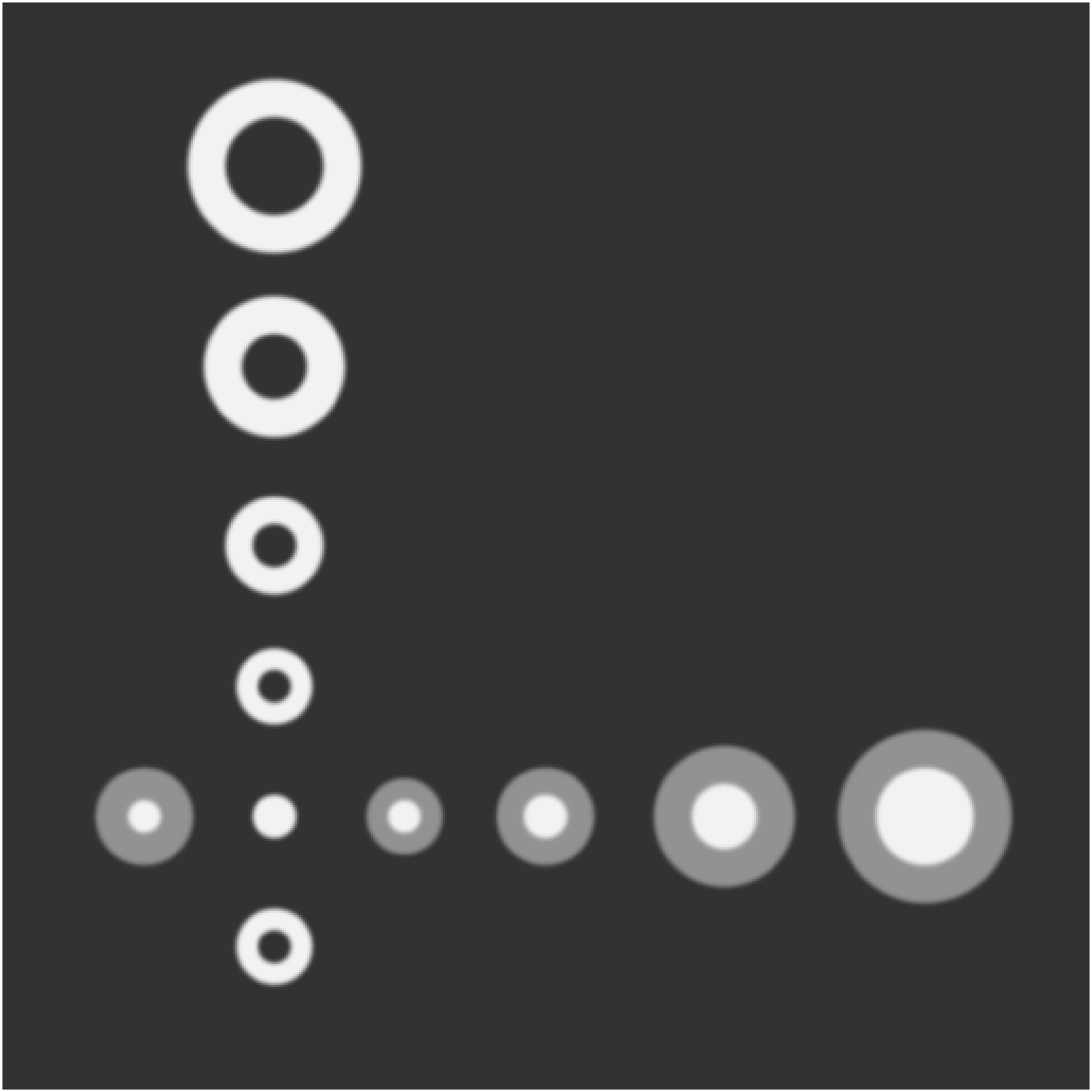} &
\includegraphics[width=2.3in,height=2.3in]{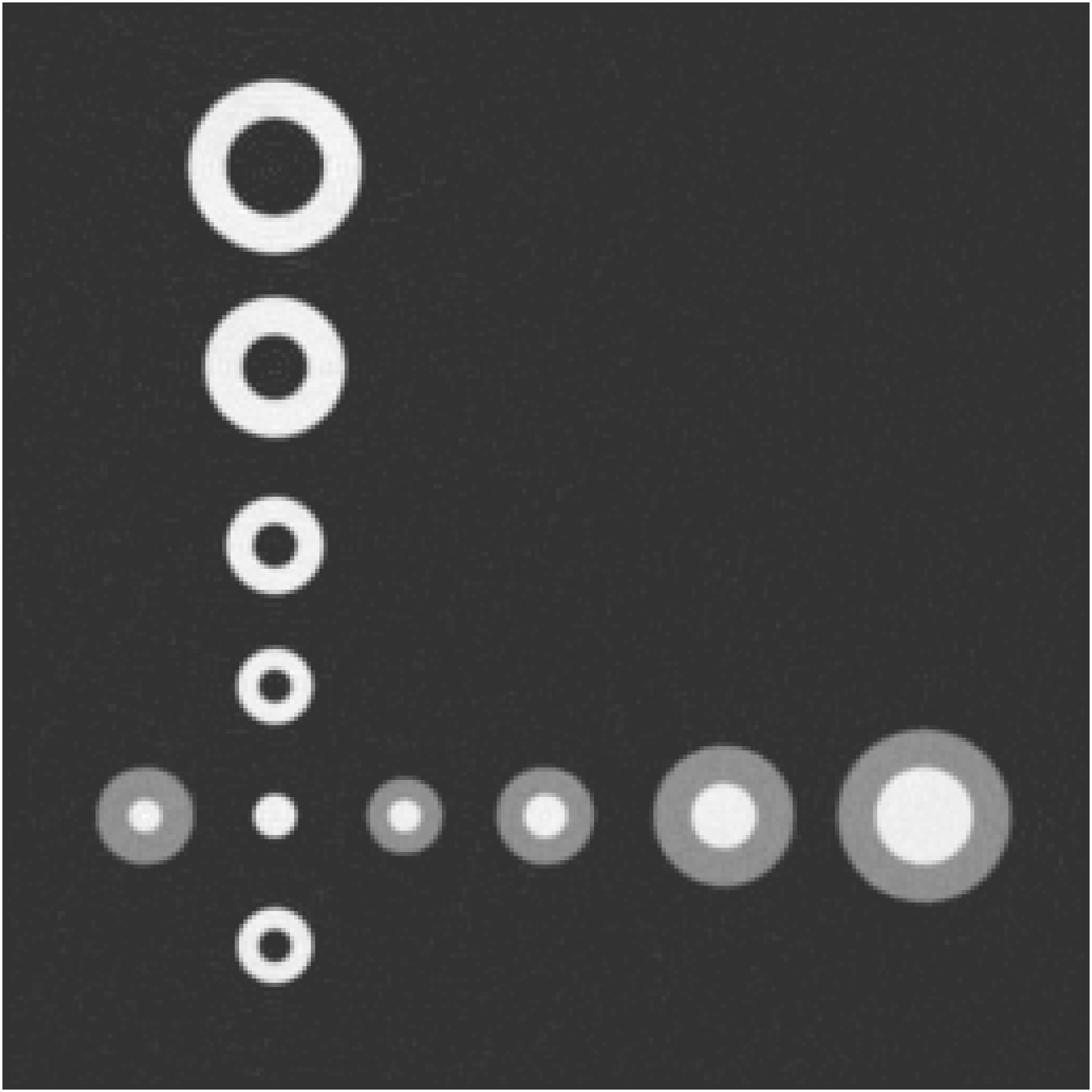}\\
(c) & (d)
\end{tabular}
\end{center}
\caption{Simulation in $3D$ with phantom shown in Figure~\ref{F:3Dphan}; (a)
and (c) are the cross sections of the phantom by planes $x_{2}=0.25$ and
$x_{1}=0.25$, respectively; (b) and (d) are the corresponding cross sections
of the reconstruction from the data with added $100\%$ (in $L_{2}$ sense)
noise}%
\label{F:rec3Dflat}%
\end{figure}

In order to better understand the origins of such unusually low
noise sensitivity, we plot in Figure~\ref{F:sigprof} a profile of one of the
time series, $M_{I_{1},B^{(3)}}(y,t)$ for point $y=(1,0.5,0.5)$. The thick
black line represents the accurate measurements, the gray line shows the with
the added noise. In Figure~\ref{F:crlprof} we plot a profile of the
reconstructed curl $C^{(1)}(x)$ (gray line) against the correct values (black
line). (This plot corresponds to the cross-section of the third component of
$C^{(1)}$ along the line $x_{2}=0.5$, $x_{3}=0.5)$. The latter figure shows
that noise is amplified during the first step of the reconstruction (inversion
of the spherical mean Radon transform). This is to be expected, since the
corresponding inverse problem is mildly ill-posed, similarly to the inversion
of the classical Radon transform. However, on the second step of the
reconstruction, corresponding to solving the problem~(\ref{E:finalsys}), the
noise is significantly smoothed out. This is not surprising, since the
corresponding operator is a smoothing one. As a result, we obtain the
low-noise image shown in Figure~\ref{F:recsmoo}(b).

The second simulation we report used an (almost) piece-wise constant phantom
of $\ln\sigma(x)$ modeled by a linear combination of several slightly smoothed
characteristic functions of balls of different radii. The centers of the balls
were located on the pair-wise intersections of planes $x_{1}=0.25$,
$x_{2}=0.25$, $x_{3}=0.25$, as shown in Figure~\ref{F:3Dphan}. The minimum
value of $\ln\sigma(x)$ in this phantom was 0 (dark black color), the maximum
value is 1 (white color). The simulated MAET data corresponded to the acoustic
sources located at the nodes of $257\times257$ Cartesian grids on each of the
six faces of cubic domain $\Omega$. For each source, a time series consisting
of 447 values for each measuring functional were simulated. In order to model
the noise, to each of the time series we added a random sequence scaled so
that the $L^{2}$ norm of the noise was equal to that of the signal (i.e.
$100\%$ noise was applied).

In Figure~\ref{F:noiseprof} we present the profile of the time series
$M_{I_{1},B^{(3)}}(y,t)$ for the point $y=(1,0.5,0.5)$. As before, the thick
black line represents the accurate measurements, and the gray line shows the
data with added noise.
\begin{figure}[t] 
\begin{center}
\includegraphics[width=3.0in,height=1.4in]{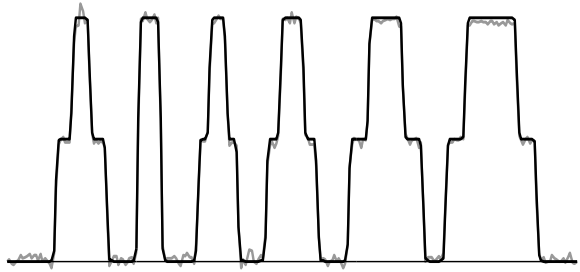}
\end{center}
\caption{ The cross section of the reconstructed image by the line
$x_{1}=0.25$, $x_{3}=0.25$. The thick black line represents the phantom, the
gray line corresponds to the image reconstructed from the data with added
$100\%$ (in $L^{2}$ sense) noise }%
\label{F:coolprof}%
\end{figure}

The reconstruction was performed on the grid of size $257\times257\times257$.
The cross sections of the reconstructed image by planes $x_{1}=0.25$ and
$x_{2}=0.25$ are shown in the Figure~\ref{F:rec3Dflat}(b) and (d), next to the
corresponding images of the phantom (i.e. parts (a) and (c) of the latter
figure). The cross section profile of the image shown in part (d),
corresponding to the line $x_{1}=0.25$, $x_{3}=0.25,$ is plotted in
Figure~\ref{F:coolprof}.

As in the first simulation, we obtain a very accurate reconstruction with
little noise. This is again the result of a smoothing operator applied when
the Poisson problem is solved on the last step of the algorithm. An additional
improvement in the quality of the image comes from the rather singular nature
of the second phantom. Indeed, while the noise is more or less uniformly
distributed over the volume of the cubic domain, the signal (the non-zero
$\ln\sigma$) is supported in a rather small fraction of the volume, thus
increasing the visual contrast between the noise and the signal.

\section*{Final remarks and conclusions}

\subsection*{Mathematical model}

In Section 1 we presented a mathematical model describing the MAET
measurements. In general, it agrees with the model used in
\cite{HaiderXu,Roth}. However, instead of point-wise electrical boundary
measurements we consider a more general scheme. The advantage of such an
approach is generality and ease of analysis and numerical modeling. In
particular, it contains as a partial case the pointwise measurement of
electrical potentials (reported in \cite{HaiderXu,Roth}).

Another novel element in this model is the use of velocity potentials which
allow us to simplify analysis and obtain a better understanding of the problem
at hand. We discussed in detail the case of acoustic signal presented by
propagating acoustic fronts from small sources. However, the same mathematics
can be used to model time-harmonic sources. Since the problem is linear with
respect to the velocity potential, the connection between the two problems is
through the direct and inverse Fourier transforms of the data in time.
Finally, plane wave irradiation (considered for example in \cite{Roth}) is a
partial case of irradiation by time harmonic sources, when they are located
far away from the object.

\subsection*{General reconstruction scheme}

In Section 2 we presented a general scheme for the solution of the inverse
problem of MAET obtained under the assumption of propagating spherical
acoustic fronts. (As we mentioned above, a slight modification of this scheme
would allow one to utilize time harmonic sources and plane waves instead of
the fronts we used). The scheme consists of the following steps:

\begin{enumerate}
\item Apply one of the suitable TAT reconstructions techniques to measuring
functionals $M_{I^{(k)},B^{(j)}}(y,t),$ $j,k=1,2,3,$ to reconstruct the
regular terms $M_{I^{(k)},B^{(j)}}^{\mathrm{reg}}(x,t)$ at $t=0$ and thus to
obtain the curls of $J^{(k)}.$

\item Compute currents $J^{(k)}$ from their curls (this step may require
solving the Neumann problem for the Laplace equation)

\item Find $\nabla\ln\sigma$ at each point in $\Omega$ using formula
(\ref{E:gradsol}) or by solving system of equations (\ref{E:new1}),
(\ref{E:new2}), (\ref{E:new3}).

\item Find values of $\Delta\ln\sigma$ by computing the divergence of
$\nabla\ln\sigma.$

\item Compute $\ln\sigma$ by solving the Poisson problem with the zero
Dirichlet boundary conditions.
\end{enumerate}

Theoretical properties and numerical methods for all three steps are well
known. The first step is mildly ill-posed (similar to the inversion of the
classical Radon transform), the second step is stable, and the third step is
described by a smoothing operator. Our rather informal discussion suggests
that the total reconstruction procedure is stable (it does not exhibit even
the mild instability present in classical computer tomography), and our
numerical experiments confirm this assertion. We leave a rigorous proof of
this conjecture for the future work.

\subsection*{Comparison with AEIT}

MAET is similar to AET in that it seeks to overcome the instability of EIT by
adding the ultrasound component to the electrical measurements. However, MAET
has some advantages:

\begin{enumerate}
\item The arising problem is linear and can be solved explicitly.

\item The AEIT measurements seem to produce a very weak signal; successful
acquisition of such signals in a realistic measuring configuration have not
been reported so far. The signal in MAET is stronger; in fact, first
reconstructions from real measurements have already been obtained
\cite{HaiderXu}.
\end{enumerate}

\subsection*{The case of a rectangular domain}

In Section 3 we presented a completely explicit set of formulae that yield a
series solution of the MAET problem for the case of the cubic domain. It
reduces the problem to a set of sine and cosine Fourier transforms, and thus,
it can be easily implemented using FFTs. This, in turn, results in a fast
algorithm that requires $\mathcal{O}(n^{3}\ln n)$ floating point operations to
complete a reconstructions on a $n\times n\times n$ Cartesian grid.

\subsection*{Feasibility of reconstruction using two directions of $B$}

It is theoretically possible to shorten the potentially long acquisition time
by reducing the number of different directions of $B$. If only two orthogonal
directions of magnetic field $B$ are used, only two components of a curl
$C=\nabla\times J_{I}$ will be reconstructed on the first step of our method
(say $C_{1}$ and $C_{2}).$ However, since $\operatorname{div}%
\operatorname{curl}J=0$,%
\begin{equation}
\frac{\partial}{\partial x_{3}}C_{3}=-\frac{\partial}{\partial x_{1}}%
C_{1}-\frac{\partial}{\partial x_{2}}C_{2}.\nonumber
\end{equation}
Since $C$ vanishes on $\partial\Omega$, the above equation can be integrated
in $x_{3},$ and thus $C_{3}$ can be reconstructed from $C_{1}$ and $C_{2}$. A
further study is needed to see how much this procedure would affect the
stability of the whole method.

\subsection*{Acknowledgements}

The author gratefully acknowledges support by the NSF through the DMS grant 0908208.

\newpage

\section*{Appendix}

Consider the following system of linear equations%
\begin{equation}
\left\{
\begin{array}
[c]{c}%
X\cdot(A\times B)=R_{1}\\
X\cdot(A\times C)=R_{2}\\
X\cdot(B\times C)=R_{3}%
\end{array}
\right.  \label{E:mat0}%
\end{equation}
where $A$, $B$, and $C$ are given linearly independent vectors from
$\mathbb{R}^{3}$, $X\in\mathbb{R}^{3}$ is the vector of unknowns, and $R_{j}$,
$j=1,2,3$, are given numbers. The first equation can be re-written in the
following form%
\begin{equation}
R_{1}=\left\vert
\begin{array}
[c]{ccc}%
x_{1} & x_{2} & x_{3}\\
a_{1} & a_{2} & a_{3}\\
b_{1} & b_{2} & b_{3}%
\end{array}
\right\vert =\left\vert
\begin{array}
[c]{ccc}%
x_{1} & a_{1} & b_{1}\\
x_{2} & a_{2} & b_{2}\\
x_{3} & a_{3} & b_{3}%
\end{array}
\right\vert =\left\vert
\begin{array}
[c]{ccc}%
x_{1} & b_{1} & c_{1}\\
x_{2} & b_{2} & c_{2}\\
x_{3} & b_{3} & c_{3}%
\end{array}
\right\vert \nonumber
\end{equation}
or%
\begin{equation}
r_{1}=\frac{1}{\det M}\left\vert
\begin{array}
[c]{ccc}%
x_{1} & b_{1} & c_{1}\\
x_{2} & b_{2} & c_{2}\\
x_{3} & b_{3} & c_{3}%
\end{array}
\right\vert ,\label{E:mat1}%
\end{equation}
where $M$ is a $(3\times3)$ matrix whose columns are vectors $A$, $B$, and
$C$, and $r_{1}=R_{1}/\det M$. Similarly,%
\begin{equation}
r_{2}=\frac{1}{\det M}\left\vert
\begin{array}
[c]{ccc}%
a_{1} & x_{1} & c_{1}\\
a_{2} & x_{2} & c_{2}\\
a_{3} & x_{3} & c_{3}%
\end{array}
\right\vert ,\label{E:mat2}%
\end{equation}
and%
\begin{equation}
r_{3}=\frac{1}{\det M}\left\vert
\begin{array}
[c]{ccc}%
x_{1} & b_{1} & c_{1}\\
x_{2} & b_{2} & c_{2}\\
x_{3} & b_{3} & c_{3}%
\end{array}
\right\vert ,\label{E:mat3}%
\end{equation}
where $r_{2}=-R_{2}/\det M$ and $r_{3}=R_{3}/\det M$. Formulae (\ref{E:mat1}%
)-(\ref{E:mat3}) can be viewed as the solution of the following system of
equations obtained using Cramer's rule:%
\begin{equation}
\left(
\begin{array}
[c]{ccc}%
a_{1} & b_{1} & c_{1}\\
a_{2} & b_{2} & c_{2}\\
a_{3} & b_{3} & c_{3}%
\end{array}
\right)  \left(
\begin{array}
[c]{c}%
r_{3}\\
r_{2}\\
r_{1}%
\end{array}
\right)  =\left(
\begin{array}
[c]{c}%
x_{1}\\
x_{2}\\
x_{3}%
\end{array}
\right)  .\nonumber
\end{equation}
Therefore, solution of system (\ref{E:mat0}) is given by the formula%
\begin{equation}
X=\frac{1}{\det M}\left(
\begin{array}
[c]{ccc}%
a_{1} & b_{1} & c_{1}\\
a_{2} & b_{2} & c_{2}\\
a_{3} & b_{3} & c_{3}%
\end{array}
\right)  \left(
\begin{array}
[c]{c}%
R_{3}\\
-R_{2}\\
R_{1}%
\end{array}
\right)  .\nonumber
\end{equation}
In addition, $\det M=A\cdot(B\times C)$.

$\bigskip$

$\bigskip$

\bigskip

\newpage



\begin{thebibliography}{99}                                                                                               %


\bibitem {AK}M.~Agranovsky and P.~Kuchment, Uniqueness of reconstruction and
an inversion procedure for thermoacoustic and photoacoustic tomography with
variable sound speed, \emph{Inverse Problems} \textbf{23} (2007) 2089--102.

\bibitem {AmbPatch}G.~Ambartsoumian and S.~Patch, Thermoacoustic tomography:
numerical results. Proceedings of SPIE 6437 \emph{ Photons Plus Ultrasound:
Imaging and Sensing 2007: The Eighth Conference on Biomedical Thermoacoustics,
Optoacoustics, and Acousto-optics}, (2007) Alexander A.~Oraevsky, Lihong V.
Wang, Editors, 64371B.

\bibitem {AmmariAET}H.~Ammari, E.~Bonnetier, Y.~Capdeboscq, M.~Tanter, and
M.~Fink, Electrical impedance tomography by elastic deformation, \emph{SIAM J.
Appl. Math.} \textbf{68} (2008) 1557--1573.

\bibitem {AmmariMAET}H.~Ammari, Y.~Capdeboscq, H.~Kang, and A.~Kozhemyak,
Mathematical models and recon\-struction methods in magneto-acoustic imaging,
\emph{Euro. Jnl. of Appl. Math.}, \textbf{20} (2009) 303--17.

\bibitem {BB1}D.~C.~Barber, B.~H.~Brown, Applied potential tomography,
\emph{J.~Phys. E.: Sci.~Instrum.} \textbf{17 }(1984), 723--733.

\bibitem {Bor02}L.~Borcea, Electrical impedance tomography, \emph{Inverse
Problems} \textbf{18} (2002) R99--R136.

\bibitem {burg-exac-appro}P.~Burgholzer, G.~J.~Matt, M.~Haltmeier, and
G.~Paltauf, Exact and approximative imaging methods for photoacoustic
tomography using an arbitrary detection surface, \emph{Phys Review E,}
\textbf{75} (2007) 046706.

\bibitem {Cap}Y.~Capdeboscq, J.~Fehrenbach, F.~de~Gournay, O.~Kavian, Imaging
by modification: numerical reconstruction of local conductivities from
corresponding power density measurements, \emph{SIAM J.~Imaging Sciences,}
\textbf{2/4} (2009) 1003--1030.

\bibitem {CIN}M.~Cheney, D.~Isaacson, and J.~C.~Newell, Electrical Impedance
Tomography, \emph{SIAM Review,} \textbf{41}, (1999) 85--101.

\bibitem {Colton}D.~Colton and R.~Kress, \emph{Inverse acoustic and
electromagnetic scattering theory}, Springer-Verlag (2001).

\bibitem {FPR}D.~Finch, S.~Patch and Rakesh, Determining a function from its
mean values over a family of spheres, \emph{SIAM J.~Math.~Anal.,} \textbf{35}
(2004) 1213--40.

\bibitem {Gilbarg}D.~Gilbarg and N.~S.~Trudinger, Elliptic Partial
Differential Equations of Second Order, Springer-Verlag (1983).

\bibitem {HaiderXu}S.~Haider, A.~Hrbek, and Y.~Xu, Magneto-acousto-electrical
tomography: a potential method for imaging current density and electrical
impedance, \emph{Physiol.~Meas.} \textbf{29} (2008) S41-S50.

\bibitem {HKN}Y.~Hristova, P.~Kuchment, and L.~Nguyen, On reconstruction and
time reversal in thermoacoustic tomography in homogeneous and non-homogeneous
acoustic media, \emph{Inverse Problems,} \textbf{24}: 055006, 2008.

\bibitem {KrugerPAT}R.~A.~Kruger, P.~Liu, Y.~R.~Fang, and C.~R.~Appledorn,
Photoacoustic ultrasound (PAUS) reconstruction tomography, \emph{Med.~Phys.},
\textbf{22 }(1995)\textbf{ }1605--09.

\bibitem {KrugerTAT}R.~A.~Kruger, D.~R.~Reinecke, and G.~A.~Kruger,
Thermoacoustic computed tomography - technical considerations, \emph{Med.
Phys.} \textbf{26} (1999) 1832--7.

\bibitem {KuKuRev}P.~Kuchment and L.~Kunyansky, Mathematics of Photoacoustic
and Thermoacoustic Tomography, Chapter 19, \emph{Handbook of Mathematical
Methods in Imaging,} Springer-Verlag, (2011) 819-865.

\bibitem {KuKuAET}P.~Kuchment and L.~Kunyansky, Synthetic focusing in
ultrasound modulated tomography, \emph{Inverse Problems and Imaging},
\textbf{4} (2010) 665 -- 673.

\bibitem {KuKuAET1}P.~Kuchment and L.~Kunyansky, 2D and 3D reconstructions in
acousto-electric tomography, \emph{Inverse Problems} \textbf{27} (2011) 055013.

\bibitem {Kunyansky}L.~Kunyansky, Explicit inversion formulae for the
spherical mean Radon transform, \emph{Inverse Problems,} \textbf{23} (2007) 737--783.

\bibitem {Kun-series}L.~Kunyansky, A series solution and a fast algorithm for
the inversion of the spherical mean Radon transform, \emph{Inverse Problems,}
\textbf{23} (2007) S11--S20.

\bibitem {Kun-cube}L.~Kunyansky, Reconstruction of a function from its
spherical (circular) means with the centers lying on the surface of certain
polygons and polyhedra, \emph{Inverse Problems, }\textbf{27} (2011) 025012.

\bibitem {Kun-cyl}L.~Kunyansky, Fast reconstruction algorithms for the
thermoacoustic tomography in certain domains with cylindrical or spherical
symmetries, Preprint (2011) Arxiv:Math.AP 1102.1413

\bibitem {Oberw}L.~Kunyansky and P.~Kuchment, Synthetic focusing in
Acousto-Electric Tomography, in \emph{Oberwolfach Report} No.~18/2010 DOI:
10.4171/OWR/2010/18, Workshop: Mathematics and Algorithms in Tomography,
Organised by Martin Burger, Alfred Louis, and Todd Quinto, April 11th -- 17th,
(2010) 44--47.

\bibitem {Lavand}B.~Lavandier, J.~Jossinet and D.~Cathignol, Experimental
measurement of the acousto-electric interaction signal in saline solution,
\emph{Ultrasonics} \textbf{38} (2000) 929--936.

\bibitem {Montalibet}A.~Montalibet, J.~Jossinet, A.~Matias, and D.~Cathignol,
Electric current generated by ultrasonically induced Lorentz force in
biological media, \emph{Med. Biol. Eng. Comput.} \textbf{39 }(2001) 15--20.

\bibitem {Norton2}S.~J.~Norton and M.~Linzer, Ultrasonic reflectivity imaging
in three dimensions: exact inverse scattering solutions for plane,
cylindrical, and spherical apertures, \emph{IEEE Trans. on Biomed. Eng.,}
\textbf{28} (1981) 200--202.

\bibitem {Nguyen}L.~Nguyen, A family of inversion formulas in thermoacoustic
tomography, \emph{Inverse Problems and Imaging}, \textbf{3} (2009) 649--675.

\bibitem {Oraev94}A.~A.~Oraevsky, S.~L.~Jacques, R.~O.~Esenaliev, and
F.~K.~Tittel, Laser-based photoacoustic imaging in biological tissues,
\emph{Proc. SPIE,} \textbf{2134A} (1994) 122--8.

\bibitem {Roth}B.~J.~Roth and K.~Schalte, Ultrasonically-induced Lorentz force
tomography, \emph{Med. Biol. Eng. Comput.} \textbf{47 }(2009) 573--7.

\bibitem {Stewart}A.~M.~Stewart, Longitudinal and transverse components of a
vector field, \emph{Classical Physics}. In press, http://arxiv.org/abs/0801.0335v2.

\bibitem {Vladimirov}V.~S.~Vladimirov, \emph{Equations of mathematical
physics}. (Translated from the Russian by Audrey Littlewood. Edited by Alan
Jeffrey.) Pure and Applied Mathematics, \textbf{3} Marcel Dekker, New York (1971).

\bibitem {WangCRC}Wang L V (Editor) 2009 \emph{"Photoacoustic imaging and
spectroscopy"} (CRC Press).

\bibitem {Wen}H.~Wen, J.~Shah, R.~S.~Balaban, Hall effect imaging, \emph{IEEE
Trans. Biomed. Eng.}, \textbf{45} (1998) 119--24.

\bibitem {MXW1}M.~Xu and L.-H.~V.~Wang, Time-domain reconstruction for
thermoacoustic tomography in a spherical geometry, \emph{IEEE Trans. Med.
Imag.,} \textbf{21} (2002) 814--822.

\bibitem {WangAET}H.~Zhang and L.~Wang, Acousto-electric tomography,
\emph{Proc. SPIE} \textbf{5320} (2004) 145--9.
\end{thebibliography}
\end{document}